\documentclass[reqno,10pt]{amsart}
\pdfoutput=1
\usepackage[utf8]{inputenc}
\usepackage{graphicx,color}
\usepackage{hyperref}
\usepackage{amssymb}
\usepackage{mathabx}
\usepackage{epstopdf}
\usepackage{tikz}
\usetikzlibrary{decorations.markings,arrows,arrows.meta,cd}

\input{definitions.tex}

\begin{document}

\date{\today}

\title{Dehn functions of subgroups of right-angled Artin groups}

\author{Noel Brady}
\address{Department of Mathematics\\
  University of Oklahoma \\
  Norman\\   OK 73019\\ USA} \email{nbrady@ou.edu}

\author{Ignat Soroko}
\address{Department of Mathematics\\
  University of Oklahoma \\
  Norman\\   OK 73019\\ USA} \email{ignat.soroko@ou.edu}


\begin{abstract}
We show that for each positive integer $k$ there exist right-angled Artin groups containing free-by-cyclic subgroups whose monodromy automorphisms grow as $n^k$. As a consequence we produce examples of right-angled Artin groups containing finitely presented subgroups whose Dehn functions grow as $n^{k+2}$.
\end{abstract}

\maketitle

\tableofcontents

\section{Introduction}

There has been intense interest in subgroups of right-angled Artin groups (RAAGs) in recent years. This is due largely to the work of Ian Agol, Dani Wise and others on the virtual fibering question in $3$--manifold topology.

In~\cite{AgolCrit} Agol showed that if $M^3$ is a compact, oriented, irreducible $3$--manifold with $\chi(M)=0$ and $\pi_1(M)$ is a subgroup of a RAAG, then $M$ virtually fibers. In~\cite{HW} Haglund and Wise showed that fundamental groups of special cubical complexes are subgroups of RAAGs. Building on this machinery, Agol went on to solve the virtual fibering conjecture in~\cite{AgolVFC}. The fundamental result in~\cite{AgolVFC} is that non-positively curved cubical complexes with hyperbolic fundamental groups are virtually special.

In this paper we consider two  questions about subgroups of RAAGs. 

The first question asks which free-by-cyclic groups virtually embed in RAAGs. In~\cite{HagW1,HagW2} Hagen and Wise show that hyperbolic free-by-cyclic groups virtually embed in RAAGs. It is also known that $F_2\rtimes\Z$ groups virtually embed in RAAGs. The hyperbolic free-by-cyclic examples all have exponentially growing monodromy automorphisms and the $F_2 \rtimes \Z$ groups have exponential or linear monodromy automorphisms. In~\cite{Ger}, Gersten gives an explicit example of an $F_3\rtimes\Z$ group which does not virtually embed in a RAAG. The group considered by Gersten is not a CAT(0) group, and this prompts the following open question. \textit{Does every CAT(0) free-by-cyclic group virtually embed in a RAAG?}  The family of the so-called Hydra groups considered in~\cite{DR} provides a test case for this question where the monodromy automorphisms grow polynomially with arbitrary degree. While we haven't proved that Hydra groups virtually embed in RAAGs (this is a topic of an ongoing research of the second-named author), we construct analogues of the Hydra groups (where the base $\Z^2$ subgroup is replaced by a more complicated RAAG) which are CAT(0) free-by-cyclic with polynomially growing monodromy automorphisms of arbitrary degree and which are virtually special. 

\begin{thma} For each positive \emph{even} integer $m$ there exist virtually special free-by-cyclic groups $G_{m,m}\cong F_{2m}\rtimes_\phi\Z$ with growth function $\gr_\phi(n)\sim n^m$ and $G_{m,m-1}\cong F_{2m-1}\rtimes_{\phi'}\Z$ with growth function $\gr_{\phi'}(n)\sim n^{m-1}$.
\end{thma}

Since finite index subgroups of free-by-cyclic groups are again free-by-cyclic and since special groups embed into RAAGs we obtain the following corollary.  

\begin{cora} For each positive integer $k$ there exist a right-angled Artin group containing a free-by-cyclic subgroup whose monodromy automorphism has growth function $\sim n^k$.
\end{cora}

The second question asks what kinds of functions arise as Dehn functions of finitely presented subgroups of RAAGs. 
Recall that Dehn functions capture the isoperimetric behavior of Cayley complexes of groups. A lot is known about Dehn functions of arbitrary finitely presented groups (see~\cite{BradyBridson,BORS,SBR}). For example, a group is hyperbolic if and only if its Dehn function is linear. CAT(0) groups and, in particular, RAAGs, have Dehn functions which are either quadratic or linear.  In~\cite{BF} there are examples of CAT(0) groups which contain finitely presented subgroups whose Dehn functions are of the form $n^\alpha$ for a dense set of $\alpha \in [2,\infty)$. Restricting to the case where the ambient groups are RAAGs it gets harder to find examples of subgroups with a wide variety of Dehn functions. In~\cite{BradyRileyShort} there are examples of finitely presented Bestvina--Brady kernels of RAAGs which have polynomial Dehn functions of degree $3$ or $4$. In \cite{Dison} it is shown that the Dehn function of such kernels of RAAGs are at most quartic. In~\cite{BriMRL} Bridson provides an example of a RAAG containing a finitely presented group with exponential Dehn function. 
Our second result shows that there are finitely presented subgroups of RAAGs whose Dehn functions are polynomial of arbitrary degree. 

\begin{thmb} For each positive integer $k$ there exists a right-angled Artin group which contains a finitely presented subgroup with Dehn function $\simeq n^k$.
\end{thmb}

This is the extent of what is currently known about the isoperimetric behavior of subgroups of RAAGs; it would be very interesting if one could produce examples with other types of Dehn functions. 

\bigskip
Our paper is organized as follows. 

In section~2 we introduce the growth functions of automorphisms and prove a folklore result (Proposition~\ref{prop:growth}) that the growth of an automorphism of a free group is invariant under taking powers of the automorphism and under passing to a subgroup of finite index. In doing that, we rely on the Gilbert Levitt's Growth Theorem~\cite{Lev} (whose proof uses train-track machinery). We also provide an example due to Yves Cornulier which demonstrates that for arbitrary groups the invariance of the growth function under passing to a subgroup of finite index does not hold.

Section~3 is devoted to providing estimates for the Dehn function of the Bieri double of a free-by-cyclic group in terms of the growth of the monodromy automorphism. We use Bridson's lower bound from~\cite{BriPoly} (Proposition~\ref{prop:lower}) and adapt the proof of the upper bound, given in~\cite{BriPit} for the abelian-by-cyclic setting, to the case of free-by-free groups needed for our construction (Proposition~\ref{prop:upper}). Using these estimates, we show later in section~7 that for the polynomially growing monodromy automorphisms involved in our construction, the upper and lower bounds on the Dehn function of the Bieri double actually coincide. If the monodromy automorphism has polynomial growth of order $n^k$, then the Dehn function of the Bieri double grows like $n^{k+2}$.

In section~4 we recollect all the relevant definitions related to the Morse theory on groups and special cubical complexes, which will be used in sections~5 and 6.

In section~5 we introduce the free-by-cyclic groups $G_{m,k}$, which play the central role in our construction. We define these groups through LOG notation, which is a graphical tool to encode conjugation relations. We prove that the group $G_{m,k}$ is CAT(0) and free-by-cyclic, and exhibit explicit formulas for its monodromy automorphism (Proposition~\ref{prop:phi}). We defer until section~8 the proof that this automorphism has growth $\sim n^k$.

The goal of section~6 is to exhibit a finite special cover for the presentation complex $K_{m,m}$ of the group $G_{m,m}$, for arbitrary even $m$. The construction is done in several stages. First, for arbitrary $m$, we engineer a certain right action of $G_{m,m}$ on a set of cardinality $2^{2m+1}$, which may be thought of as the $0$--skeleton of a $(2m+1)$--dimensional torus $\T_{2m+1}$. This action defines a finite cover $\widehat K_m\to K_{m,m}$, which cellularly embeds into the $2$--skeleton of $\T_{2m+1}$ (Proposition~\ref{prop:emb}). Since $\widehat K_m$ is a subcomplex of a product of graphs, it is free from three out of four hyperplane pathologies in the definition of a special cube complex (Proposition~\ref{prop:clean}). To eliminate the fourth hyperplane pathology we observe that for even values of $m$, the complex $K_{m,m}$ is a $VH$-complex in the terminology of~\cite{HW}. It follows that there exist another finite cover $\ov{K}_m\to \widehat K_m$, such that $\ov{K}_m$ is a special square complex (Proposition~\ref{prop:vh}).

In section~7 we bring all the pieces together and prove Theorems A and B. For even values of $m$, groups $G_{m,m}$ are virtually special free-by-cyclic with the monodromy automorphism growing as $n^m$. To obtain growth functions of odd degree, we observe that the presentation $2$--complex $K_{m,m-1}$ of the free-by-cyclic group $G_{m,m-1}$ is a combinatorial subcomplex of $K_{m,m-1}$ and it is obtained by deleting the hyperplane corresponding to the last generator $a_{2m+1}$. Thus the pullback of $K_{m,m-1}$ in $\ov{K}_m$ is a finite special square complex covering $K_{m,m-1}$. This makes $G_{m,m-1}$ a virtually special free-by-cyclic group with the monodromy automorphism growing as $n^{m-1}$.

To prove theorem B, we look at the Bieri double of the special (free-by-cyclic) finite index subgroups $H$ of $G_{m,m}$ ($G_{m,m-1}$), and prove that the lower and the upper bounds for its Dehn function coincide, and are of the order $n^{m+2}$ (resp., $n^{m+1}$). This Bieri double naturally embeds into a RAAG, whose underlying graph is the join of the underlying graph for the RAAG containing $H$ and the empty graph on two vertices.

In section~8 we provide the computation of the growth function for the monodromy automorphism of $G_{m,k}$ and its abelianization.

Finally, in section~9 we list two open questions related to the study in this paper.

\section{Preliminaries on growth}
\label{sec:prelims}

In what follows we will consider functions up to the following equivalence relations.
\begin{defi}
Two functions $f,g\colon[0,\infty)\to[0,\infty)$ are said to be $\sim$ equivalent if $f\pr g$ and $g\pr f$, where $f\pr g$ means that there exist constants $A>0$ and $B\ge0$ such that $f(n)\leq Ag(n)+B$ for all $n\geq0$. 
\end{defi}

\begin{defi}
Two functions $f,g\colon[0,\infty)\to[0,\infty)$ are said to be $\simeq$ equivalent if $f\prq g$ and $g\prq f$, where $f\prq g$ means that there exist constants $A,B>0$ and $C,D,E\ge0$ such that $f(n)\leq Ag(Bn+C)+Dn+E$ for all $n\geq0$. 
\end{defi}

We extend these equivalence relations to functions $\N\to[0,\infty)$ by assuming them to be constant on each interval $[n,n+1)$. 
\begin{rem}
Notice that $f\pr g$ implies $f\prq g$ and $f\sim g$ implies $f\simeq g$. However, the relation $\sim$ is strictly finer than $\simeq$, as the latter identifies all single exponential functions, i.e. $k^n\simeq K^n$ for $k,K>1$, whereas the relation $\sim$ does not. We will use the relations $\sim$, $\pr$ when dealing with growth functions of automorphisms and the relations $\simeq$, $\prq$ when discussing Dehn functions of groups (as is done traditionally). 
The term $Dn$ in the above definition of $\prq$ is essential for proving the equivalence of Dehn functions under quasi-isometries. 
\end{rem}

Let $F$ be a free group of finite rank $k$ with a finite generating set $\A$. Let $d_\A(x,y)$ be the associated word metric on $F$. If $\psi\colon G\to G$ is an automorphism, we define
\[
\gr_{\psi,\A}(n):=\max_{a\in\A}\|\psi^n(a)\|_\A,
\]
where $\|g\|_\A$ is equal to $d_\A(1,g)$ for $g\in F$. 

The following properties of $\gr_{\psi,\A}$ will be used in the sequel. 

\begin{prop}\label{prop:growth}
Let $F$ be a free group of finite rank with a finite generating set $\A$, and let $\psi$ be an automorphism of $F$. Then
\begin{itemize}
\item[(i)] for each finite generating set $\B$ of $F$, $\gr_{\psi,\B}\sim\gr_{\psi,\A}$;
\item[(ii)] for each $d\in\N$, $\gr_{\psi,\A}\sim \gr_{\psi^d,\A}$;
\item[(iii)] for each finite index subgroup $H\le F$ invariant under $\psi$ with a finite generating set $\B\subset H$, we have $\gr_{\psi,\A}\sim \gr_{\psi|_H,\B}$.
\end{itemize}
\end{prop}
\begin{proof}
Let $\A=\{a_1,\dots,a_N\}$, $\B=\{b_1,\dots,b_M\}$. Since both sets $\A$ and $\B$ generate $F$, there exist words $w_i$ and $v_j$ such that $a_i=w_i(b_1,\dots,b_M)$ and $b_j=v_j(a_1,\dots,a_N)$, for all $1\le i\le N$, $1\le j\le M$. Let constants $K$ and $L$ denote the maximal lengths of $w_i$, $v_j$, respectively, i.e. $K=\max_{1\le i\le N}\|w_i\|_\B$, $L=\max_{1\le j\le M}\|v_j\|_\A$. Then, obviously, for all $i,j,n$, one has:
\[
\|\psi^n(a_i)\|_\B\le K\cdot \|\psi^n(a_i)\|_\A\text{\qquad and \qquad}\|\psi^n(b_j)\|_\A\le L\cdot \|\psi^n(b_i)\|_\B.
\]
Now fix arbitrary $1\le j\le M$ and assume without loss of generality that \\ 
$v_j(a_1,\dots, a_N)=a_{i_1}^{\eps_1}\dots a_{i_L}^{\eps_L}$, for values $1\le i_\ell\le N$ and $\eps_\ell=\pm1$ or $0$. Then
\begin{multline*}
\|\psi^n(b_j)\|_\B=
\|\psi^n(a_{i_1}^{\eps_1}\dots a_{i_L}^{\eps_L})\|_\B\le
\|\psi^n(a_{i_1})\|_\B+\dots+\|\psi^n(a_{i_L})\|_\B \le \\
K\|\psi^n(a_{i_1})\|_\A+\dots+K\|\psi^n(a_{i_L})\|_\A \le
KL\max_{1\le i\le N}\|\psi^n(a_i)\|_\A = KL \gr_{\psi,\A}(n).
\end{multline*}
Therefore, $\gr_{\psi,\B}(n)=\max_{1\le j\le M}\|\psi^n(b_j)\|_\B\le KL \gr_{\psi,\A}(n)$, and, by symmetry, $\gr_{\psi,\A}(n)\le LK \gr_{\psi,\B}(n)$. This proves (i).

Before proving parts (ii) and (iii), we state following remarkable result of Gilbert Levitt:
\begin{lgthm}[{\cite[Cor.~6.3]{Lev}}]
Let $F$ be a free group of finite rank with a free generating set $\A$. Given $\alpha\in\Aut(F)$ and $g\in F$, there exist $\lambda\ge1$, an integer $m\ge0$ and constants $A$, $B>0$ such that the word length $\|\alpha^n(g)\|_\A$ satisfies:
\[
\pushQED{\qed}
\forall n\in\N,\qquad A\lambda^nn^m\le\|\alpha^n(g)\|_\A\le B\lambda^nn^m.\qedhere
\popQED
\] 
\end{lgthm}

Consider a sequence of growth parameters $(\lambda_i,m_i)$ from the Levitt's Growth Theorem corresponding to the generators $a_1,\dots,a_N$, so that for each $1\le i\le N$ 
there exist constants $A_{i},B_{i}>0$ such that 
\[
A_{i}\lambda_i^nn^{m_i}\le \|\psi^n(a_i)\|_\A\le B_{i}\lambda_i^nn^{m_i}\text{\qquad for all $n\in\N$.}
\]
Order these parameters lexicographically: $(\lambda_i,m_i)<(\lambda_j,m_j)$ if and only if $\lambda_i<\lambda_j$ or $\lambda_i=\lambda_j$ and $m_i<m_j$. Clearly, $(\lambda_i,m_i)<(\lambda_j,m_j)$ if and only if 
$\lambda_j^nn^{m_j}/\lambda_i^nn^{m_i}\to \infty$ as $n\to\infty$.
Pick $1\le i_0\le N$ such that $(\lambda_{i_0},m_{i_0})$ is maximal with respect to this order. Then for any constants $C_1, C_2>0$ and arbitrary $1\le i\le N$ we have: 
\[
C_1\lambda_i^nn^{m_i}\ll C_2\lambda_{i_0}^nn^{m_{i_0}},
\]
which means that the left-hand side is less than or equal to the right-hand side for all large enough $n\in\N$.

To prove (ii) in one direction, notice that for any $1\le i\le N$,
\begin{multline*}
\|(\psi^d)^n(a_i)\|_\A\le B_i\lambda_i^{dn}(dn)^{m_i}=(B_i\lambda_i^d d^{m_i})\lambda_i^n n^{m_i}\ll
\tfrac{B_i\lambda_i^dd^{m_i}}{A_{i_0}}\big(A_{i_0}\lambda_{i_0}^nn^{m_{i_0}}\big)\le \\
\tfrac{B_i\lambda_i^dd^{m_i}}{A_{i_0}}
\|\psi^n(a_{i_0})\|_\A.
\end{multline*}
Hence, there exist a constant $C_{big}\ge0$ such that 
\[
\gr_{\psi^d,\A}(n)=\max_{1\le i\le N}\|(\psi^d)^n(a_i)\|_\A\le D\gr_{\psi,\A}(n)+C_{big},
\]
where $D=\max_{1\le i\le N}{B_i\lambda_i^dd^{m_i}}/{A_{i_0}}$. Thus, $\gr_{\psi^d,\A}\pr\gr_{\psi,\A}$.

In the opposite direction, for any $1\le i\le N$ we have:
\[
\|\psi^n(a_i)\|_\A \le B_{i}\lambda_i^nn^{m_i} \le B_{i}\lambda_{i}^{dn}(dn)^{m_i}\ll \tfrac{B_i}{A_{i_0}}\big(A_{i_0}\lambda_{i_0}^{dn}(dn)^{m_{i_0}}\big) \le \tfrac{B_i}{A_{i_0}}\|\psi^{dn}(a_{i_0})\|_\A.
\]
By taking maximum, we get for arbitrary $n\in\N$:
\[
\gr_{\psi,\A}(n)=\max_{1\le i\le N}\|\psi^n(a_i)\|_\A \le D\|\psi^{dn}(a_{i_0})\|_\A+C_{big}\le D\gr_{\psi^d,\A}(n)+C_{big}
\]
for $D={\max_iB_i}/{A_{i_0}}$ and some $C_{big}\ge0$. This proves that $\gr_{\psi,\A}\pr \gr_{\psi^d,\A}$ and hence that 
$\gr_{\psi,\A}\sim \gr_{\psi^d,\A}$.

To prove (iii) in one direction, notice first that $H$, being of finite index in $F$, is 
quasi-convex in $F$ (see e.g.~\cite[III.3.5]{BH}).
Hence for any $h\in H$, one has $\|h\|_\B\le C\|h\|_\A$ for some $C>0$. Writing each $b_j\in\B$ as a word $b_j=v_j(a_1,\dots,a_N)$ and setting $L=\max_{1\le j\le M}\|v_j\|_\A$, we obtain for arbitrary $1\le j\le M$:
\[
\|\psi^n(b_j)\|_\B\le C\|\psi^n(b_j)\|_\A\le CL\max_{1\le i\le N} \|\psi^n(a_i)\|_\A=CL \gr_{\psi,\A}(n),
\]
so that $\gr_{\psi|_H,\B}(n)=\max_{1\le j\le M}\|\psi^n(b_j)\|_\B\le CL \gr_{\psi,\A}(n)$, i.e. $\gr_{\psi|_H,\B}\pr\gr_{\psi,\A}$.

In the opposite direction, notice that there exist an integer $p>0$ such that for every $a_i\in\A$, we have $a_i^p\in H$. As above, let 
$(\lambda_i,m_i)$, $A_i,B_i>0$ be a sequence of growth parameters for the generators $a_1,\dots,a_N$, and let $(\lambda_{i_0}, m_{i_0})$ be maximal. Consider a new generating set $\B'$ for $H$, $\B'=\B\cup\{a_{i_0}^p\}$. Then for arbitrary $1\le i\le N$ 
we have:
\begin{multline*}
\|\psi^n(a_i)\|_\A \le B_{i}\lambda_i^nn^{m_i} \ll \tfrac{B_i}{A_{i_0}}\big(A_{i_0}\lambda_{i_0}^nn^{m_{i_0}}\big)\le 
\tfrac{B_i}{A_{i_0}}\|\psi^n(a_{i_0})\|_\A \le \\
\tfrac{B_i}{A_{i_0}}\|\psi^n(a_{i_0}^p)\|_\A\le \tfrac{B_i}{A_{i_0}}L\|\psi^n(a_{i_0}^p)\|_{\B'}, 
\end{multline*}
where $L$ has a similar meaning as above. Here the fourth inequality holds since, in general, for any automorphism $\alpha\in\Aut(F)$, any $g\in F$ and any $p>0$, one has $\|\alpha(g)\|_\A\le\|\alpha(g^p)\|_\A$. (Indeed, one can write $\alpha(g)=uvu^{-1}$ with $v$ cyclically reduced. Then $\|\alpha(g)\|=2\|u\|+\|v\|$, whereas $\|\alpha(g^p)\|=\|uv^pu^{-1}\|=2\|u\|+p\|v\|$.) 

By taking maximum, we get for arbitrary $n\in\N$:
\[
\gr_{\psi,\A}(n)=\max_{1\le i\le N}\|\psi^n(a_i)\|_\A\le D L\|\psi^n(a_{i_0}^p)\|_{\B'}+C_{big}\le DL\gr_{\psi|_H,\B'}(n)+C_{big}
\]
for $D={\max_iB_i}/{A_{i_0}}$ and some $C_{big}\ge0$. This means that $\gr_{\psi,\A}\pr\gr_{\psi|_H,\B'}$. Since, by part~(i), $\gr_{\psi|_H,\B'}\sim\gr_{\psi|_H,\B}$, this proves that
$\gr_{\psi,\A}\pr \gr_{\psi|_H,\B}$ and 
part (iii) is proved.
\end{proof}

\begin{rem} 
The proof of property (i) of Proposition~\ref{prop:growth} works for automorphisms of arbitrary finitely generated groups.
\end{rem}
\begin{rem}
Property (iii) of Proposition~\ref{prop:growth} does not hold for arbitrary finitely generated groups. We are grateful to Yves Cornulier for providing the following example. Let $G=\la a,b\mid aba^{-1}=b^{-1},a^2=1\ra$ be the infinite dihedral group. Then the inner automorphism $i_b\colon x\mapsto bxb^{-1}$ has linear growth, but its restriction to the index $2$ subgroup $\la b\ra$ is trivial. To see that, observe that $aba^{-1}=b^{-1}$ implies $ab=b^{-1}a$ and hence $ba=ab^{-1}$. Thus $bab^{-1}=ab^{-2}$ and $i_b^n(a)=b^nab^{-n}=ab^{-2n}$. Looking at the Cayley graph of $G$ shows that the element $g=ab^{-2n}$ is at distance $2n+1$ from $1$, so that $ab^{-2n}$ is a word of minimal length representing element $g$, and $i_b$ indeed grows linearly on $G$. 
\end{rem}

In view of item (i) in Proposition~\ref{prop:growth}, we will suppress the dependence on the generating set and adopt the notation
\[
\gr_{\psi}(n):=\gr_{\psi,\A}(n),
\]
for an arbitrary generating set $\A\subset F$.

\bigskip
For the abelianization $F_{ab}=F/[F,F]\cong\Z^k$ we consider the induced automorphism ${\psi^{ab}}\colon F_{ab}\to F_{ab}$ and denote $\{\bei\}$ be the generating set of $F_{ab}$ corresponding to $\A$: $\bei=a_i[F,F]$, $a_i\in\A$, $i=1,\dots,k$. For any $v\in\Z^k$ let $|v|_1$ denote the $\ell_1$-norm on $\Z^k$ viewed as a subset of $\C^k$: if $v=\sum_{i=1}^kc_i\bei$, then $|v|_1=\sum_{i=1}^k|c_i|$. Define
\[
\gr_{{\psi^{ab}}}(n):=\max_{i=1,\dots,k}|(\psi^{ab})^n(\bei)|_1.
\]
Then the following is true:
\begin{lem}\label{lem7}
\[
\gr_\psi(n)\geq \gr_{{\psi^{ab}}}(n).
\]
\end{lem}
\begin{proof}
Let $\eps\colon F\to F_{ab}$ be the natural homomorphism. Then
\[
(\psi^{ab})^n(\bei)=\eps(\psi^n(a_i))
\]
and hence the length of the shortest word in generators $\{\bei\}_{i=1}^k$ of the element $(\psi^{ab})^n(\bei)\in\Z^k$ is no bigger than $\|\psi^n(a_i)\|_\A$. But the former is equal to $|(\psi^{ab})^n(\bei)|_1$ hence $|(\psi^{ab})^n(\bei)|_1\leq \|\psi^n(a_i)\|_\A$ for all $i=1,\dots,k$. By taking maximum, we get the required inequality.
\end{proof}

By embedding $\Z^k$ into $\C^k$ we may consider $\C^k$ as a vector space with the basis $\{\bei\}_{i=1}^k$.
Now let $A$ be a linear operator on $\C^k$ given in basis $\{\bei\}_{i=1}^k$ by the matrix $(a_{ij})_{i,j=1}^k$, and let $|v|_\infty$ denote the $\ell_\infty$-norm on $\C^k$: if $v=\sum_{i=1}^kc_i\bei$, then $|v|_\infty=\max_{i=1,\dots,k}|c_i|$. Consider two norms on $\End(\C^k)$, one is the operator norm with respect to $\ell_\infty$:
\[
\|A\|_{op}=\sup_{v\ne0}\frac{|Av|_\infty}{|v|_\infty}=\sup_{|v|_\infty=1}|Av|_\infty,
\]
and another one is the supremum norm, which is the $\ell_\infty$-norm on the space $\C^{k^2}$:
\[
\|A\|_{sup}=\max_{i,j=1,\dots,k}|a_{ij}|.
\]

\begin{lem}\label{lem8}
\[
\max_{i=1,\dots,k}|A\bei|_1\geq \|A\|_{sup}.
\]
\end{lem}
\begin{proof}
\[
\max_i|A\bei|_1\geq\max_i|A\bei|_\infty=\max_{i,j}|a_{ij}|=\|A\|_{sup}.\qedhere
\]
\end{proof}

The following fact is well-known (see~\cite[Cor.~5.4.5]{HJ}):
\begin{lem}\label{lem9}
There exist constants $C_1$, $C_2>0$ such that
\[
\pushQED{\qed}
C_1\|A\|_{op}\leq\|A\|_{sup}\leq C_2\|A\|_{op}.\qedhere
\popQED
\]
\end{lem}

\begin{cor}\label{cor3}
The growth function
\[
\gr_A^{sup}\colon n\longmapsto \|A^n\|_{sup}
\]
is $\sim$ equivalent to the growth function
\[
\pushQED{\qed}
\gr_A^{op}\colon n\longmapsto \|A^n\|_{op}.\qedhere
\popQED
\]
\end{cor}

The following results are proved in~\cite[Proof of Th.\,2.1]{BG}:
\begin{lem}\label{lem10}
\pushQED{\qed}
The $\sim$ equivalence class of the function $\gr_A^{op}$ depends only on the conjugacy class of $A$ in $GL(k,\C)$.\qedhere
\popQED
\end{lem}

In view of Corollary~\ref{cor3} and Lemma~\ref{lem10}, we need only to consider the growth of the Jordan normal forms of matrices $A$.

\begin{lem}[{\cite[Th.\,2.1]{BG}}]\label{lem11}\pushQED{\qed}
Suppose that $J$ is a matrix in the Jordan normal form with all eigenvalues equal to $1$. Then $\gr_J^{sup}(n)\sim n^{c-1}$, where $c$ is the maximal size of Jordan blocks of $J$.\qedhere
\popQED
\end{lem}

Combining all of the above, we get:
\begin{cor}\label{cor4}
Let $\psi$ be an automorphism of a free group $F$. 
If the abelianization ${\psi^{ab}}$ has all eigenvalues equal to $1$, and $c$ is the size of the largest Jordan block in the Jordan normal form $J$ for ${\psi^{ab}}$, then
\[
\gr_\psi(n)\su n^{c-1}\text{\quad and\quad}\gr_{{\psi^{ab}}}(n)\su n^{c-1}.
\]
\end{cor}
\begin{proof}
Indeed,
\begin{align*}
\gr_\psi(n)& \geq 
									\gr_{{\psi^{ab}}}(n) \quad & & \text{(by Lemma~\ref{lem7})}\\
    & \geq \|(\psi^{ab})^n\|_{sup} \quad & & \text{(by Lemma~\ref{lem8})}\\
    & \geq C_1\|(\psi^{ab})^n\|_{op} \quad & & \text{(for some $C_1>0$, by Lemma~\ref{lem9})}\\
    & \sim \|J^n\|_{op} \quad & & \text{(by Lemma~\ref{lem10})}\\
    & \sim n^{c-1}\quad & & \text{(by Lemma~\ref{lem11})}.\qedhere
\end{align*}
\end{proof}

\section{Bounding the Dehn function of the Bieri double\label{sec:dehn}}

In this section we outline what is known about the upper and the lower bounds for the Dehn function of a Bieri double group. The lower bound was established in~\cite[Lemma 1.5]{BriPoly} (see Proposition~\ref{prop:lower} below). The argument for the upper bound (see Proposition~\ref{prop:upper} below) follows the outline of~\cite[Theorem 5.1]{BriPit}. In the latter paper the argument is given in the setting of abelian-by-cyclic groups; we adapt this reasoning to the free-by-free setting.

\begin{defi} (Bieri double)
Let $G$ be a free-by-cyclic group $G=F\rtimes_\psi\Z$. The {\it Bieri double} of $G$ is the group $\Gamma(G)=G\Asterisk_FG$.
\end{defi}
If $G\cong\la\A,t\mid tat^{-1}=\psi(a)\text{ for all }a\in\A\ra$ then $\Gamma(G)\cong\la\A,s,t\mid sas^{-1}=\psi(a),\,tat^{-1}=\psi(a)\text{ for all }a\in\A\ra$. If one denotes $F(s,t)$ the free group on the generating set $\{s,t\}$ and $(\psi)\colon F(s,t)\to\Aut(F(\A))$ the homomorphism given on the generators by $s\mapsto\psi$, $t\mapsto\psi$, then $\Gamma(G)\cong F(\A)\rtimes_{(\psi)}F(s,t)$.

\begin{defi} (Dehn function)
Let a group $\Gamma$ be given by a finite presentation $P=\la\A\mid R \ra$. For each word $w$ lying in the normal closure of $R$ in the free group $F(\A)$, define 
\[
\Area(w):=\min\big\{N\,\left|\right.\, w\underset{\scriptscriptstyle{F(\A)}}{=}\prod_{i=1}^N x_i^{-1}r_ix_i\text{ with }x_i\in F(\A), r_i\in R^\pm\big\}.
\]
The \emph{Dehn function} of $P$ is the function $\delta_P\colon\N\to\N$ defined by
\[
\delta_P(n):=\max\{\Area(w)\mid w\underset{\scriptscriptstyle\Gamma}{=}1, \|w\|_\A\le n\}.
\]
where $\|w\|_\A$ denotes the length of the word $w$ in generators $\A^\pm$.
\end{defi}

Viewed up to $\simeq$ equivalence, the Dehn functions are independent of the choice of the presentation (see~\cite[1.3.3]{BriChap}), so we denote $\delta_P(n)$ as $\delta_\Gamma(n)$.

\begin{prop}[{\cite[Lemma 1.5]{BriPoly} and \cite[Proposition~7.2.2]{BriChap}}]\label{prop:lower}
Let $\psi$ be an automorphism of $F$ and $\|.\|$ denote the word length with respect to a fixed generating set of $F$. Then for the Dehn function $\delta_\Gamma(n)$ of the Bieri double $\Gamma$ of $F\rtimes_\psi\Z$ one has 
\[
\pushQED{\qed}
n\cdot\max_{\substack{\|b\|\leq n\\ b\in F}}\|\psi^n(b)\|\prq \delta_\Gamma(n).\qedhere
\popQED
\]
\end{prop}

\begin{prop}\label{prop:upper}
Let $\psi$ be an automorphism of a free group $F$ and assume that $\gr_\psi(n)\pr n^d$ and $\gr_{\psi^{-1}}(n)\pr n^d$. Then for the Dehn function $\delta_\Gamma(n)$ of the Bieri double $\Gamma=\Gamma(F\rtimes_\psi\Z)$ one has $\delta_\Gamma(n)\prq n^{d+2}$.
\end{prop}

\begin{rem}
It can be proved using train-tracks that $\gr_{\psi^{-1}}(n)\sim n^d$ if and only if $\gr_{\psi}(n)\sim n^d$ (see e.g.~\cite[Th.~0.4]{Pig}). However, for the reader who is unfamiliar with the train-track machinery we make the exposition independent of this result. Instead, in what follows we will apply Proposition~\ref{prop:upper} to the automorphisms $\phi$ whose growth functions $\gr_{\phi}(n)$ and $\gr_{\phi^{-1}}(n)$ are computed in section~\ref{sect:growth} and are shown to be $\sim$ equivalent to each other.
\end{rem}

In order to prove Proposition~\ref{prop:upper}, we need some preliminary results on combings of groups. We start with some definitions from~\cite{BriGFA} and \cite{BriPit}.

Let $\Gamma$ be a group with finite generating set $\A$ and $d_\A(x,y)$ be the associated word metric.
\begin{defi}
A \emph{combing} (normal form) for $\Gamma$ is a set of words $\{\sigma_g\mid g\in\Gamma\}$ in the letters $\A^\pm$ such that $\sigma_g=g$ in $\Gamma$. We denote by $|\sigma_g|$ or $|\sigma_g|_\A$ the length of the word $\sigma_g$ in the free monoid on $\A^\pm$.
\end{defi}

\begin{defi}
Let
\[
\mathcal R=\{\rho\colon\N\to\N\mid\rho(0)=0;\,\rho(n+1)\in\{\rho(n),\rho(n)+1\}\,\forall n;\,\rho\text{ unbounded\,}\}.
\]
Given eventually constant paths $p_1,p_2\colon\N\to(\Gamma,d)$ we define
\[
D(p_1,p_2)=\min_{\rho,\rho'\in \mathcal R}\big\{\max_{t\in\N}\{d_\A(p_1(\rho(t)),p_2(\rho'(t))\}\big\}.
\]
\end{defi}
\begin{defi}
Given a combing $\sigma$ for $\Gamma$, the \emph{asynchronous width} of $\sigma$ is the function $\Phi_\sigma\colon \N\to\N$ defined by
\[
\Phi_\sigma(n)=\max\big\{\,D(\sigma_g,\sigma_h)\mid d_\A(1,g),d_\A(1,h)\leq n;\,d_\A(g,h)=1\,\big\}.
\]
\end{defi}
\begin{defi}
A finitely generated group $\Gamma$ is said to be \emph{asynchronously combable} if there exists a combing $\sigma$ for $\Gamma$ and a constant $K>0$ such that $\Phi_\sigma(n)\leq K$ for all $n\in\N$.
\end{defi}
\begin{defi}
The \emph{length} of a combing $\sigma$ for $\Gamma$ is the function $L\colon\N\to\N$ given by:
\[
L(n)=\max\big\{\,|\sigma_g|\mid d_\A(1,g)\leq n\,\big\}.
\]
\end{defi}

The relation of combings to Dehn functions is manifested in the following result:
\begin{prop}[{\cite[Lemma 4.1]{BriPit}}]\label{propBriPit}
Let $\Gamma$ be a group with a finite set of semigroup generators $\A^\pm$. 
If there exists a combing $\sigma$ for $\Gamma$ whose asynchronous width is bounded by a constant and whose length is bounded by the function $L(n)$, then the Dehn function $\delta_\Gamma(n)$ for any presentation of $\Gamma$ satisfies $\delta_\Gamma(n)\prq nL(n)$.\qed
\end{prop}

In connection to the groups which are Bieri doubles, the following result from~\cite{BriGFA} is useful.
\begin{thm}[{\cite[Theorem B]{BriGFA}}]\label{th:1}
If $G$ is word-hyperbolic and $H$ is asynchronously combable then every split extension 
\[
1\longrightarrow G \longrightarrow G\rtimes H \longrightarrow H \longrightarrow 1
\]
of $G$ by $H$ is asynchronously combable.\qed
\end{thm}
\begin{rem}
From the proof of this result in~\cite{BriGFA} it follows that if groups $G$ and $H$ have combings $\sigma^G$ and $\sigma^H$ whose asynchronous width is bounded by some constants, then the combing for the split extension $G\rtimes H$ of $G$ by $H$, whose length is bounded by a constant, can be taken as the product (concatenation) $\sigma^H\sigma^G$ of combings $\sigma^H$ and $\sigma^G$, meaning that we traverse path $\sigma^H$ first, then path $\sigma^G$. Note that the product of combings in the opposite order, $\sigma^G\sigma^H$, may not have bounded asynchronous width, as the example of Baumslag--Solitar groups shows.
\end{rem}

Now let again $\Gamma=F\rtimes_{(\psi)}F(s,t)$ be the Bieri double of $G=F\rtimes_\psi\Z$, where $F$ is a free group on the set of free generators $\A$.

Our goal is to obtain an upper bound on the length $L(n)$ of the combing $\sigma^{F(s,t)}\cdot\sigma^{F(\A)}$ in terms of the growth of the automorphism $\psi$. (Here we treat a combing on a free group as a unique reduced word in a fixed system of generators which represents the given element of the group.) We prove the following proposition, adapting the reasoning for the abelian-by-cyclic groups from~\cite[Theorem 5.1]{BriPit} to the case of free-by-free groups.
\begin{prop}\label{prop-comb}
Let $P(n)$ be an increasing function bounding the growth of both $\psi$ and $\psi^{-1}$, i.e. $d_{\A}(1,\psi^n(a))\leq P(|n|)$ for all $a\in \A$, $n\in\Z$. Then the length $L(n)$ of the combing $\sigma^{F(s,t)}\cdot \sigma^{F(\A)}$ of the group $\Gamma=F({\A})\rtimes_{(\psi)}F(s,t)$ satisfies
\[
L(n)\leq nP(n)+n.
\]
\end{prop}
\begin{proof}
Take arbitrary $\gamma\in\Gamma$ and write it as $\gamma=u\cdot g$, where $u\in F(s,t)$, $g\in F(\A)$. Let $n_0=d_{\A\cup\{s,t\}}(1,\gamma)$ be the length of the shortest word in generators $(\A\cup\{s,t\})^\pm$ representing element $\gamma$ in $\Gamma$. We would like to show that
\[
|\sigma_\gamma|=|\sigma_u\cdot\sigma_g|=d_{\{s,t\}}(1,u)+d_{\A}(1,g)\leq n_0P(n_0)+n_0.
\]
Considering the natural homomorphism $\eta\colon F(\A)\rtimes_{(\psi)}F(s,t)\to F(s,t)$, one observes that $u=\eta(\gamma)$ and hence $d_{\{s,t\}}(1,u)\leq n_0$. Therefore it suffices to show that
\[
d_{\A}(1,g)\leq n_0P(n_0).
\]

Denote $w_0$ the shortest word in generators $(\A\cup\{s,t\})^\pm$ such that $w_0=\gamma$ in $\Gamma$, so that $|w_0|_{\A\cup\{s,t\}}=n_0$. Then $w_0$ can be written as
\[
w_0=u_1w_1u_2w_2\cdot\dots\cdot u_rw_r,
\]
where $u_i\in F(s,t)$, $w_i\in F(\A)$ for all $i$. Then
\[\tag{*}
n_0=|w_0|_{\A\cup\{s,t\}}=\sum_{i=1}^r|u_i|_{\{s,t\}}+ \sum_{i=1}^r|w_i|_{\A}
\]
and $u=u_1\dots u_r$.

Denote 
\[
v_i=\prod_{j=1}^iu_j\cdot w_i\cdot\Big(\prod_{j=1}^iu_j\Big)^{-1}, \qquad i=1,\dots,r.
\]
Then, as one easily checks,
\[
v_1v_2\dots v_r=u_1w_1u_2w_2\dots u_rw_r\cdot \Big(\prod_{j=1}^iu_j\Big)^{-1}
\]
so that
\[
\gamma=w_0=v_1v_2\dots v_r\cdot \Big(\prod_{j=1}^iu_j\Big).
\]
Hence
\begin{multline*}
g=u^{-1}\gamma=\Big(\prod_{j=1}^ru_j\Big)^{-1}\cdot v_1v_2\dots v_r\cdot \Big(\prod_{j=1}^ru_j\Big)=\\
\prod_{i=1}^r\Big[\Big(\prod_{j=1}^ru_j\Big)^{-1}\cdot \prod_{j=1}^iu_j\cdot w_i\cdot \Big(\prod_{j=1}^iu_j\Big)^{-1}\cdot \prod_{j=1}^ru_j\Big]=\\
\prod_{i=1}^r\big(u_r^{-1}u_{r-1}^{-1}\dots u_{i+1}^{-1}\big)\cdot w_i\cdot \big(u_{i+1}\dots u_r\big).
\end{multline*}
If we denote by $\eps\colon F(s,t)\to\Z$ the homomorphism defined on the generators as: $s\mapsto 1$, $t\mapsto 1$, then for any $g\in F(\A)$ and any $u\in F(s,t)$ we have $ugu^{-1}=\psi^{\eps(u)}(g)$. Therefore,
\[
g=\prod_{i=1}^r\psi^{-\eps(u_{i+1}\dots u_r)}(w_i)=\prod_{i=1}^r\psi^{-\sum_{j=i+1}^r\eps(u_j)}(w_i).
\]

On the other hand,
\[\tag{**}
\Big|\sum_{j=i+1}^r\eps(u_j)\Big|\leq\sum_{j=i+1}^r|\eps(u_j)|\leq\sum_{j=1}^r|\eps(u_j)|\leq\sum_{j=1}^r|u_j|_{\{s,t\}}=|u|_{\{s,t\}}\leq n_0
\]
by the observation above.

Moreover, since for any $a\in\A$ we have $d_{\A}(1,\psi^n(a))\leq P(|n|)$, then for any $w_i\in F(\A)$ we get
\[
d_{\A}\big(1,\psi^n(w_i)\big)\leq P(|n|)\cdot d_{\A}(1,w_i).
\]

Finally, we get for the element $g$ the estimate:
\begin{multline*}
d_{\A}\big(1,g\big)\leq\sum_{i=1}^rd_{\A}\big(1,\psi^{-\sum_{j=i+1}^r\eps(u_j)}(w_i)\big)
\leq\sum_{i=1}^rP\Big(\Big|\sum_{j=i+1}^r\eps(u_j)\Big|\Big)\cdot d_{\A}(1,w_i)\\
\leq[\text{by~(**)}]\leq\sum_{i=1}^rP(n_0)\cdot d_{\A}(1,w_i)=P(n_0)\cdot\sum_{i=1}^rd_{\A}(1,w_i)\leq
[\text{by (*)}]\leq P(n_0)n_0.
\end{multline*}
This shows that $|\sigma_\gamma|\leq n_0P(n_0)+n_0$ and finishes the proof of the Proposition.
\end{proof}

Now we are ready to prove the upper bound for the Dehn function of the Bieri double.

\begin{proof}[Proof of Proposition~\ref{prop:upper}] 
As was noted above, $\Gamma=\Gamma(F\rtimes_\psi\Z)\cong F\rtimes_{(\psi)}F(s,t)$. As a free group, $F$ is asynchronously combable (with constant $K=1$) and $F(s,t)$ is also word-hyperbolic. Therefore by Theorem~\ref{th:1}, $\Gamma$ is asynchronously combable and hence, by Proposition~\ref{propBriPit}, $\delta_\Gamma(n)\prq n L(n)$. But due to Proposition~\ref{prop-comb}, $L(n)\prq n^{d+1}$, and therefore $\delta_\Gamma(n)\prq n^{d+2}$.
\end{proof}

\section{Cube complexes}

\subsection{Special cube complexes}

In their article~\cite{HW} Haglund and Wise established that the fundamental groups of the so-called special cube complexes admit embeddings into right-angled Artin groups. This gives us a natural class of subgroups of right-angled Artin groups and suggests that we construct our examples within this class. We summarize the relevant definitions and results from~\cite{HW}  about special cube complexes in this section.

\begin{defi} (Cube complex) An \emph{$n$--cube} is a copy of $[0,1]^n\subset \R^n$, viewed as a metric space with the euclidean metric of $\R^n$. (We will suppress `$n$--' and call `$n$--cubes' just `cubes'.) A \emph{face} is a metric subspace of an $n$--cube obtained by restricting some of coordinates (or all of them) to either $0$ or $1$. A \emph{cube complex} is a CW complex obtained by gluing $n$--cubes (of possibly varied dimensions $n$) together along faces via isometries. If all cubes of a cube complex are $2$--cubes, such cube complex is called a \emph{square complex}. A cube complex is \emph{simple} if the link of every vertex of is a simplicial complex. A simplicial complex is \emph{flag} if any collection of $k+1$ pairwise adjacent vertices spans a $k$--simplex. A cube complex is \emph{non-positively curved} if the link of each vertex is a flag simplicial complex. 
\end{defi}

\begin{defi} (Hyperplane)
A \emph{midcube} of an $n$--cube $[0,1]^n$ is a subset obtained by restricting one of the coordinates to $\frac12$. A \emph{hyperplane} of a cube complex $X$ is a connected component of a new cube complex $Y$ which is formed as follows:
\begin{itemize}
\item the cubes of $Y$ are the midcubes of $X$;
\item the restriction of a $(k+1)$--cell of $X$ to a midcube of $[0,1]^k$ defines the attaching map of a $k$--cell in $Y$.
\end{itemize}
An edge $a$ of $X$ is \emph{dual} to some hyperplane $H$ if the midpoint of $a$ is a vertex of $H$.
\end{defi}

\begin{defi} (Parallelism, Walls)
Two oriented edges $a$, $b$ of a cube complex $X$ are called \emph{elementary parallel} if there is a square of $X$ containing $a$ and $b$ and such that the attaching map sends two opposite edges of $[0,1]\times[0,1]$ with the same orientation to $a$ and $b$ respectively. Define the \emph{parallelism} on oriented edges of $X$ as the equivalence relation generated by elementary parallelism. An \emph{(oriented) wall} of $X$ is a parallelism class of oriented edges. Note that every hyperplane $H$ in $X$ defines a pair of oriented walls consisting of edges dual to $H$.
\end{defi}

Now we describe four pathologies for interaction of hyperplanes in a cube complex, which are forbidden for special cube complexes.

\begin{defi} (Self-intersection) A hyperplane $H$ in $X$ \emph{self-intersects}, if it contains more than one midcube from the same cube of $X$. 
\end{defi}

\begin{defi} (One-sided) 
A hyperplane $H$ is \emph{two-sided} if there exists a combinatorial map of CW complexes $H\times[0,1]\to X$ mapping $H\times\{\frac12\}$ identically to $H$. (Recall that a cellular map $f\colon X\to Y$ of CW complexes is \emph{combinatorial} if the restriction of $f$ to each open cell of $X$ is a homeomorphism onto its image.) A hyperplane $H$ in $X$ is called \emph{one-sided} if it is not two-sided.
\end{defi}

\begin{defi} (Self-osculating) A hyperplane $H$ in $X$ is \emph{self-osculating} if there are two edges $a$, $b$ dual to $H$ which do not belong to a common square of $X$ but share a common vertex. If in addition there is a consistent choice of orientation on the edges dual to $H$ which makes the common vertex for $a$, $b$ their origin or terminus, then the hyperplane $H$ is called \emph{directly self-osculating}.
\end{defi}

\begin{defi} (Inter-osculating) Two distinct hyperplanes $H_1$, $H_2$ of $X$ are \emph{inter-oscula\-ting} if they intersect and there are edges $a_1$ dual to $H_1$ and $a_2$ dual to $H_2$ which do not belong to the same square of $X$ but share a common vertex.
\end{defi}

\begin{defi} (Special cube complex) A non-positively curved cube complex is called \emph{special} if its hyperplanes are all two-sided, with no self-intersections, self-osculations or inter-osculations.
\end{defi}

\begin{defi} (Virtually special group) A group $G$ is called \emph{special} if there exists a special cube complex $X$ whose fundamental group is isomorphic to $G$. A group $G$ is \emph{virtually special} if there exists a special cube complex $X$ and a finite index subgroup $H\le G$ such that $H$ is isomorphic to the fundamental group of $X$. 
\end{defi}

\begin{defi} (Right-angled Artin group)
Let $\Delta$ be a simplicial graph. The \emph{right-angled Artin group}, or RAAG, associated to $\Delta$, is a finitely presented group $A(\Delta)$ given by the presentation:
\[
A(\Delta)=\la a_i\in\Vertices(\Delta)  \mid [a_i,a_j]=1 \text{ if } (a_i,a_j)\in \Edges(\Delta)\ra.
\] 
\end{defi}

\begin{defi} (Salvetti complex)
Given a right-angled Artin group $A(\Delta)$, the \emph{Salvetti complex} associated to $A(\Delta)$ is a non-positively curved cube complex $S_\Delta$ defined as follows. For each $a_i\in\Vertices(\Delta)$ let $S^1_{a_i}$ be a circle endowed with a structure of a CW complex having a single $0$--cell and a single $1$--cell. Let $n=\Card(\Vertices(\Delta))$ and let $T=\prod_{j=1}^n S^1_{a_{j}}$ be an $n$--dimensional torus with the product CW structure. For every full subgraph $K\subset\Delta$ with $\Vertices(K)=\{a_{i_1},\dots,a_{i_k}\}$ define a $k$--dimensional torus $T_K$ as a Cartesian product of CW complexes:
$T_K=\prod_{j=1}^k S^1_{a_{i_j}}$
and observe that $T_K$ can be identified as a combinatorial subcomplex of $T$. Then the \emph{Salvetti complex associated with $A(\Delta)$} is 
\[
S_\Delta=\bigcup\big\{T_K\subset T\mid K \text{ a full subgraph of }\Delta\big\}.
\]
Thus $S_\Delta$ has a single $0$--cell and $n$ $1$--cells. Each edge $(a_i,a_j)\in\Edges(\Delta)$ contributes a square $2$--cell to $S_\Delta$ with the attaching map $a_ia_ja_i^{-1}a_j^{-1}$. And in general each full subgraph $K\subset \Delta$ contributes a $k$--dimensional cell to $S_\Delta$ where $k=\Card(\Vertices(K))$.
\end{defi}

\begin{thm}[\cite{HW},Th.~4.2] A cube complex is special if and only if it admits a local isometry into the Salvetti complex of some right-angled Artin group.
\end{thm}
Since local isometries of CAT(0) spaces are $\pi_1$-injective, one gets the following

\begin{cor}\label{cor:spec_raag}
The fundamental group of a special cube complex is isomorphic to a subgroup of a right-angled Artin group.
\end{cor}

\subsection{Morse theory for cube complexes}

We will use the following definitions from~\cite{BaBr,BeBr}.

\begin{defi} (Morse function)
A map $f\colon X\to\R$ defined on a cube complex $X$ is a \emph{Morse function} if 
\begin{itemize}
\item for every cell $e$ of $X$, with the characteristic map $\chi_e\colon [0,1]^m\to e$, the composition $f\chi_e\colon [0,1]^m\to\R$ extends to an affine map $\R^m\to \R$ and $f\chi_e$ is constant only when $\dim e=0$;
\item the image of the $0$--skeleton of $X$ is discrete in $\R$.
\end{itemize}
\end{defi}

\begin{defi} (Circle-valued Morse function)
A \emph{circle-valued Morse function} on a cube complex $X$ is a cellular map $f\colon X\to S^1$ with the property that $f$ lifts to a Morse function between universal covers.
\end{defi}

\begin{defi} (Ascending and descending links)
Suppose $X$ is a cube complex, $f\colon X\to S^1$ is a circle-valued Morse function and $\tilde f\colon \tilde X\to\R$ is the corresponding Morse function. Let $v\in X^{(0)}$ and note that the link of $v$ in $X$ is naturally isomorphic to the link of any lift $\tilde v$ of $v$ in $\tilde X$. We say that a cell $\tilde e\subset \tilde X$ contributes to the \emph{ascending} (respectively \emph{descending}) link of $\tilde v$ if $\tilde v\in\tilde e$ and $\tilde f|_{\tilde e}$ achieves its minimum (respectively, maximum) value at $\tilde v$. The ascending (respectively, descending) link of $v$ is then defined to be the subset of the link $\Lk(v,X)$ naturally identified with the ascending (respectively, descending) link of $\tilde v$. Note that in the case when $X$ is a square complex, all ascending, descending and entire links are graphs.
\end{defi}

The following characterization of free-by-cyclic groups was proven in~\cite{BaBr} (see also~\cite[Th.~10.1]{Howie}).

\begin{thm}[\cite{BaBr}, Proposition 2.5]\label{thm:fbc}
If $f\colon X\to S^1$ is a circle-valued Morse function on a $2$--complex $X$ all of whose ascending and descending links are trees, then X is aspherical and $\pi_1(X)$ is free-by-cyclic. 
This means that there is a short exact sequence 
\[
1\longrightarrow F_m\longrightarrow \pi_1(X)\longrightarrow \Z\longrightarrow 1,
\]
where the free group $F_m$ is isomorphic to $\pi_1(f^{-1}(\mathrm{pt}))$, $\mathrm{pt}$ being any point on $S^{1}$.
\end{thm}

\section{Groups $G_{m,k}$}

In this section we define a sequence of groups $G_{m,k}$ and study their presentation complex. We show that it is a non-positively curved square complex and that the groups are free-by-cyclic. 

\subsection{LOG definition\label{sect:log}} Recall that a {\em labeled, oriented graph}, or LOG, consists of a finite, directed graph with labels on the vertices and edges satisfying the following: the vertices have distinct labels, and the edge labels are chosen from the set of vertex labels. 

A LOG determines a  finite presentation as follows. The set of generators is the set of vertex labels. The set of relations is in one-to-one correspondence with the set of edges;  there is a relation of the form $a^{-1}ua=v$ for each oriented edge labeled $a$ from vertex $u$ to vertex $v$.  

Let $m\in\N$, $m\ge1$. For $k=0,\dots,m$, let $G_{m,k}$ be a group defined by the LOG presentation in the Figure~\ref{fig:log}:

\begin{figure}[htb] 
\begin{tikzpicture}[scale=1.0]
\begin{scope}[thick]
\fill (0,0) circle (1.5pt) node[right=1pt] {$a_{m+1}$};
\fill (0,2) circle (1.5pt) node[below right=0.5pt] {$a_{1}$};
\fill (2,0) circle (1.5pt) node[right=1pt] {$a_{m+2}$};
\fill (2,2) circle (1.5pt) node[below right=0.5pt] {$a_{2}$};
\fill (5,0) circle (1.5pt) node[right=1pt] {$a_{m+k}$};
\fill (5,2) circle (1.5pt) node[below right=0.5pt] {$a_{k}$};
\fill (7,0) circle (1.5pt) node[right=1pt] {$a_{m+k+1}$};
\fill (7,2) circle (1.5pt) node[below right=0.5pt] {$a_{k+1}$};
\fill (10,2) circle (1.5pt) node[below right=0.5pt] {$a_{m}$};

\draw[->-=0.53] (0,3) ellipse (0.5cm and 1cm) node[left=-4pt] {$a_2$};
\draw[->-=0.57] (0,2)-- node[right=1pt] {$a_{m+2}$} (0,0);
\draw[->-=0.53] (2,3) ellipse (0.5cm and 1cm) node[left=-4pt] {$a_3$};
\draw[->-=0.57] (2,2)-- node[right=1pt] {$a_{m+3}$} (2,0);
\draw (3.5,3) node {{\bf\ldots}};
\draw[->-=0.53] (5,3) ellipse (0.5cm and 1cm) node[left=-14pt] {$a_{k+1}$};
\draw[->-=0.57] (5,2)-- node[right=1pt] {$a_{m+k+1}$} (5,0);
\draw[->-=0.53] (7,3) ellipse (0.5cm and 1cm) node[left=-14pt] {$a_{k+2}$};
\draw (8.5,3) node {{\bf\ldots}};
\draw[->-=0.53] (10,3) ellipse (0.5cm and 1cm) node[left=-16pt] {$a_{m+1}$};

\end{scope}
\end{tikzpicture}
\caption{\label{fig:log} The LOG description of $G_{m,k}$.} 
\end{figure}
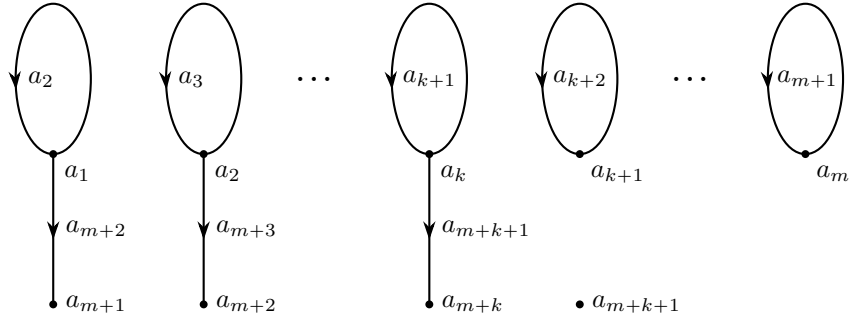

\noindent
i.e. 
\begin{multline*}
G_{m,k}=\la\, a_1,\dots,a_{m+k+1} \mid [a_i,a_{i+1}]=1,\quad i=1,\dots,m,\\
a_{m+j+1}^{-1}a_ja_{m+j+1}=a_{m+j},\quad j=1,\dots,k\, \ra.
\end{multline*}

Clearly $G_{m,k}$ is an HNN extension of $G_{m,k-1}$ with the stable letter $a_{m+k+1}$ so there is a natural tower of inclusions 
\[
G_{m,0}\subset G_{m,1}\subset G_{m,2}\subset\dots \subset  G_{m,m}.
\]

\subsection{CAT$(0)$ structure for $G_{m,k}$}

\begin{figure}[h!]
\begin{tikzpicture}[scale=0.75]

\fill (0,0) circle (2pt);
\fill (0,3) circle (2pt);
\fill (3,0) circle (2pt);
\fill (3,3) circle (2pt);

\begin{scope}[]
\draw[->-=0.53] (0,0)--node[left=1pt] {$a_{j+1}$}(0,3);
\draw[->-=0.54] (0,3)--node[above=1pt] {$a_j$}(3,3);
\draw[->-=0.54] (0,0)--node[below=1pt] {$a_j$} (3,0);
\draw[->-=0.53] (3,0)--node[right=1pt] {$a_{j+1}$}(3,3);
\end{scope}

\begin{scope}[densely dashed]
\draw (0.6,0)--(0,0.6);
\draw (0,2.4)--(0.6,3);
\draw (2.4,3)--(3,2.4);
\draw (3,0.6)--(2.4,0);
\end{scope}


\begin{scope}[thick]
\fill (7,0.5) circle (2pt) node[left=1pt] {$a_j^-$};
\fill (7,2.5) circle (2pt) node[left=1pt] {$a_j^+$};
\fill (10,0.5) circle (2pt) node[right=1pt] {$a_{j+1}^-$};
\fill (10,2.5) circle (2pt) node[right=1pt] {$a_{j+1}^+$};
\draw (7,0.5)--(10,0.5);
\draw (7,2.5)--(10,2.5);
\draw (7,0.5)--(10,2.5);
\draw (7,2.5)--(10,0.5);
\end{scope}


\fill (0,-5) circle (2pt);
\fill (0,-2) circle (2pt);
\fill (3,-5) circle (2pt);
\fill (3,-2) circle (2pt);

\begin{scope}[]
\draw[->-=0.53] (0,-5)--node[left=1pt] {$a_{m+j+1}$}(0,-2);
\draw[->-=0.54] (0,-2)--node[above=1pt] {$a_{m+j}$}(3,-2);
\draw[->-=0.54] (0,-5)--node[below=1pt] {$a_j$} (3,-5);
\draw[->-=0.53] (3,-5)--node[right=1pt] {$a_{m+j+1}$}(3,-2);
\end{scope}

\begin{scope}[densely dashed]
\draw (0.6,-5)--(0,-4.4);
\draw (0,-2.6)--(0.6,-2);
\draw (2.4,-2)--(3,-2.6);
\draw (3,-4.4)--(2.4,-5);
\end{scope}


\begin{scope}[thick]
\fill (7,-5) circle (2pt) node[left=1pt] {$a_j^-$};
\fill (7,-4) circle (2pt) node[left=1pt] {$a_j^+$};
\fill (7,-3) circle (2pt) node[left=1pt] {$a_{m+j}^-$};
\fill (7,-2) circle (2pt) node[left=1pt] {$a_{m+j}^+$};
\fill (10,-4.5) circle (2pt) node[right=1pt] {$a_{m+j+1}^-$};
\fill (10,-2.5) circle (2pt) node[right=1pt] {$a_{m+j+1}^+$};
\draw (7,-5)--(10,-4.5);
\draw (7,-4)--(10,-4.5);
\draw (7,-3)--(10,-2.5);
\draw (7,-2)--(10,-2.5);
\end{scope}

\end{tikzpicture}
\caption{The contribution of the relations $a_{j+1}^{-1}a_ja_{j+1}=a_j$, $1 \leq j \leq m$, and $a_{m+j+1}^{-1}a_ja_{m+j+1}=a_{m+j}$, $1 \leq j \leq k$, to the link of the $0$--cell of the presentation complex $K_{m,k}$.}
\label{fig:lk1}
\end{figure}
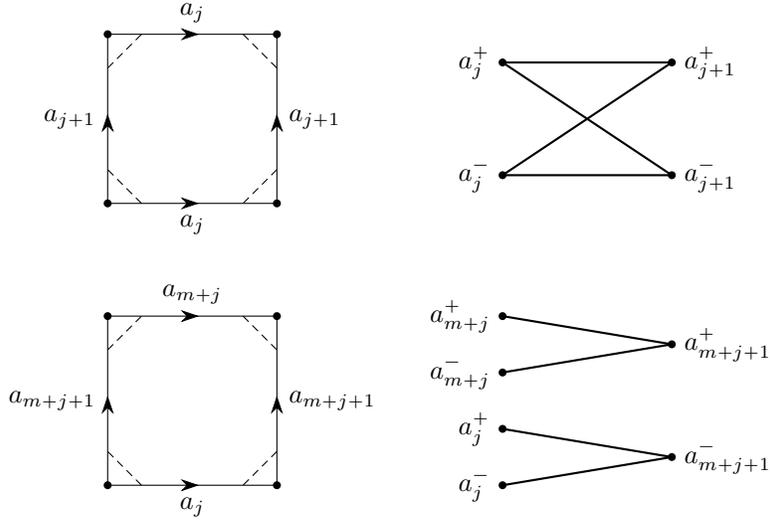

One way of producing a CAT$(0)$ structure on groups $G_{m,k}$ is to verify that the presentation $2$--complex corresponding to their LOG presentation can be metrized so that it is a non-positively curved, piecewise euclidean (PE) complex. 

Let $K_{m,k}$ denote the presentation $2$--complex corresponding to the LOG presentation above of $G_{m,k}$. It has one $0$--cell, $(m+k+1)$ $1$--cells labeled by $a_1, \ldots, a_{m+k+1}$, and $(m+k)$  $2$--cells corresponding to the relations $a_{j+1}^{-1}a_ja_{j+1}=a_j$ for $1 \leq j \leq m$ and $a_{m+j+1}^{-1}a_{j}a_{m+j+1} = a_{m+j}$ for $1 \leq j \leq k$. By construction, $K_{m,k}$ is a subcomplex of $K_{m,k+1}$. We endow $K_{m,k}$ with a PE structure by using regular euclidean squares for the $2$--cells, and using local isometric embedding attaching maps. 

\begin{prop}\label{prop:cat0}
The presentation complex $K_{m,k}$ defined above is a non-positively curved PE complex.
\end{prop}
\begin{proof}
We need to check the Gromov link condition~\cite[Th.~II.5.20]{BH}. For the square $2$--cells, it reduces to a purely combinatorial requirement that the link of every $0$--cell has no circuits of combinatorial length less than $4$. Figure~\ref{fig:lk1} shows the contributions of the relations of $G_{m,k}$ to the link $L$ of the unique $0$--cell of $K_{m,k}$. We adopt the following notation: if a $1$--cell $a$ originates at $0$--cell $u$ and terminates at $0$--cell $v$, then it contributes a vertex denoted $a^-$ to the link of $u$ and a vertex denoted $a^+$ to the link of $v$.

We see that the link $L$ can be obtained as a union of a sequence of graphs:
\[
L_1\subset L_2\subset\dots\subset L_{m+k+1}=L,
\]
where $L_1$ is just a pair of disjoint vertices $a_1^+$, $a_1^-$, and $L_{i+1}$ is obtained from $L_i$ by adding a new pair of disjoint vertices $a_{i+1}^+$, $a_{i+1}^-$ and connecting each one of them to some pair $a_s^+$, $a_s^-$ with $s<i+1$. We observe that this procedure preserves the following property: ``for every $l$, vertices $a_l^+$, $a_l^-$ are non-adjacent''. Indeed, the shortest path between the ``old'' vertices $a_s^+$, $a_s^-$ has length two, and the shortest path between the newly added vertices $a_{i+1}^+$, $a_{i+1}^-$ is at least two. This shows that at each step we cannot create cycles of lengths two and three. Therefore the link $L$ has no cycles of length less than four.
\end{proof}

\begin{cor}
Groups $G_{m,k}$ are CAT$(0)$.
\end{cor}
\begin{proof}
Indeed, the universal cover $\widetilde K_{m,k}$ of non-positively curved square complex $K_{m,k}$ is a CAT($0$) complex and $G_{m,k}$ acts on it by isometries, properly discontinuously and cocompactly.
\end{proof}

\subsection{Free-by-cyclic structure}

Notice that all the relations of groups $G_{m,k}$ have the form: $a_i^{a_j}=a_l$. This implies that there exists a well-defined epimorphism $G_{m,k}\to\Z$, sending every $a_i$ to a fixed generator of $\Z$. This epimorphism can be realized geometrically by a circle-valued Morse function $f\colon K_{m,k}\to S^1$, which can be defined as follows. Consider a CW structure on $S^1$ consisting of one $0$--cell and one $1$--cell. Then $f$ takes the $0$--cell of $K_{m,k}$ to the $0$--cell of $S^1$, maps $1$--cells of $K_{m,k}$ map homeomorphically onto the target $1$--cell of $S^1$, and extends linearly over the $2$--cells. Here by `extends linearly' we mean that $f$ lifts to a map of the universal covers in the way depicted in the Figure~\ref{fig:morse}. (Note that, by the non-positive curvature, characteristic maps of cells lift to embeddings in the universal cover.)

\begin{figure}[!h]
\begin{tikzpicture}[scale=1.25]

\begin{scope}[]

\fill (1,0) circle (4/3pt);
\fill (1,2) circle (4/3pt);
\fill (0,1) circle (4/3pt);
\fill (2,1) circle (4/3pt);
\draw[->-=0.53] (1,0)--node[below left=1pt] {$a_{i+1}$}(0,1);
\draw[->-=0.53] (1,0)--node[below right=1pt] {$a_{i}$}(2,1);
\draw[->-=0.53] (0,1)--node[above left=1pt] {$a_{i}$}(1,2);
\draw[->-=0.53] (2,1)--node[above right=1pt] {$a_{i+1}$}(1,2);
\draw[->-=0.53,densely dashed] (0,1)--node[above=1pt] {$A_i$}(2,1);

\fill (4,0) circle (4/3pt);
\fill (4,2) circle (4/3pt);
\fill (3,1) circle (4/3pt);
\fill (5,1) circle (4/3pt);
\draw[->-=0.53] (4,0)--node[below left=1pt] {$a_{m+j+1}$}(3,1);
\draw[->-=0.53] (4,0)--node[below right=1pt] {$a_{j}$}(5,1);
\draw[->-=0.53] (3,1)--node[above left=1pt] {$a_{m+j}$}(4,2);
\draw[->-=0.53] (5,1)--node[above right=1pt] {$a_{m+j+1}$}(4,2);
\draw[->-=0.53,densely dashed] (3,1)--node[above=1pt] {$B_j$}(5,1);
\end{scope}
\draw[mapping=0.99] (6,1)--node[above=1pt] {$\tilde f$}(7,1);
\draw (8,2.5) node[right=1pt]{$\R$};
\draw (8,-0.5)--(8,2.5);
\fill (8,0) circle (1pt) node[right=1pt] {$-1$};
\fill (8,1) circle (1pt) node[right=1pt] {$0$};
\fill (8,2) circle (1pt) node[right=1pt] {$1$};

\end{tikzpicture}
\caption{\label{fig:morse} The Morse function on each $2$--cell and the preimage set of $0$.}
\end{figure}
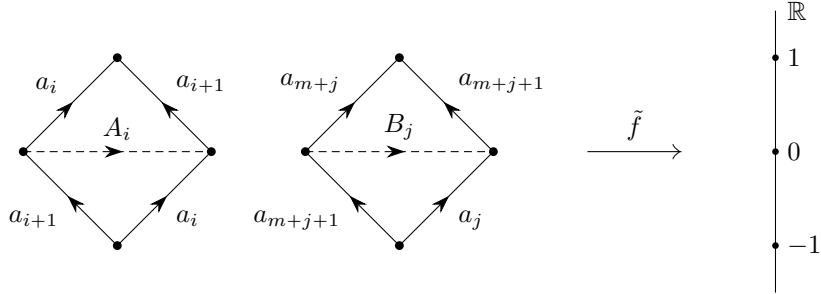

\begin{prop}\label{prop:morse}
The (circle-valued) Morse function $f\colon K_{m,k}\to S^1$ induces a short exact sequence
\[
1\longrightarrow F_{m+k}\longrightarrow G_{m,k}\longrightarrow \Z\longrightarrow 1,
\]
where $F_{m+k}$ is a free group of rank $m+k$.
\end{prop}
\begin{proof}
By Theorem~\ref{thm:fbc} it suffices to show that the ascending and the descending links of the $0$--cell in $K_{m,k}$ are trees.  

The ascending link of the $0$--cell of $K_{m,k}$ is formed by those corners of $2$--cells of $K_{m,k}$ which are formed by a pair of originating edges (labeled $a_\bullet^-$ in Figure~\ref{fig:lk1}). Similarly, the descending link of the $0$--cell of $K_{m,k}$ is formed by those corners of $2$--cells of $K_{m,k}$ which are formed by a pair of terminating edges (labeled $a_\bullet^+$, in Figure~\ref{fig:lk1}).

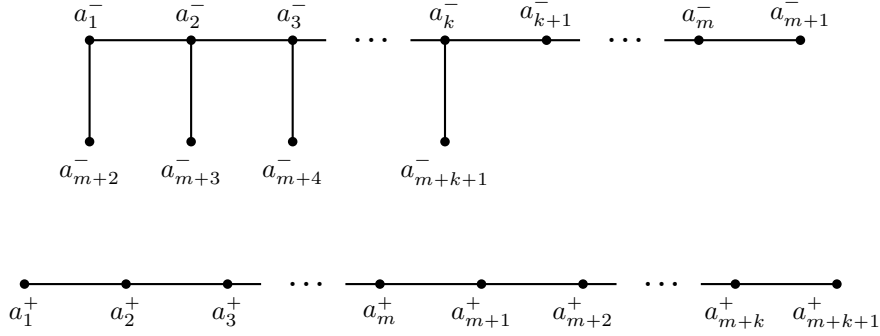
\begin{figure}[h!]
\begin{tikzpicture}[scale=1.35]

\begin{scope}[thick]
\fill (0,1) circle (4/3pt) node[above=1pt] {$a_1^-$};
\fill (1,1) circle (4/3pt) node[above=1pt] {$a_2^-$};
\fill (2,1) circle (4/3pt) node[above=1pt] {$a_3^-$};
\fill (0,0) circle (4/3pt) node[below=1pt] {$a_{m+2}^-$};
\fill (1,0) circle (4/3pt) node[below=1pt] {$a_{m+3}^-$};
\fill (2,0) circle (4/3pt) node[below=1pt] {$a_{m+4}^-$};
\draw (0,1)--(1,1);
\draw (1,1)--(2,1);
\draw (0,1)--(0,0);
\draw (1,1)--(1,0);
\draw (2,1)--(2,0);
\draw (2,1)--(2.33,1);
\draw (2.8,1) node {{\bf\ldots}};
\draw (3.16,1)--(3.5,1);
\fill (3.5,1) circle (4/3pt) node[above=1pt] {$a_{k}^-$};
\fill (3.5,0) circle (4/3pt) node[below=1pt] {$a_{m+k+1}^-$};
\draw (3.5,1)--(3.5,0);
\fill (4.5,1) circle (4/3pt) node[above=1pt] {$a_{k+1}^-$};
\draw (3.5,1)--(4.5,1);
\draw (4.5,1)--(4.83,1);
\draw (5.3,1) node {{\bf\ldots}};
\draw (5.66,1)--(6,1);
\fill (6,1) circle (4/3pt) node[above=1pt] {$a_{m}^-$};
\draw (6,1)--(7,1);
\fill (7,1) circle (4/3pt) node[above=1pt] {$a_{m+1}^-$};

\end{scope}
\end{tikzpicture}

\vspace{1cm}

\begin{tikzpicture}[scale=1.35]

\begin{scope}[thick]

\fill (0,0) circle (4/3pt) node[below=1pt] {$a_1^+$};
\draw (0,0)--(1,0);
\fill (1,0) circle (4/3pt) node[below=1pt] {$a_2^+$};
\draw (1,0)--(2,0);
\fill (2,0) circle (4/3pt) node[below=1pt] {$a_3^+$};
\draw (2,0)--(2.33,0);
\draw (2.8,0) node {{\bf\ldots}};
\draw (3.16,0)--(3.5,0);
\fill (3.5,0) circle (4/3pt) node[below=1pt] {$a_{m}^+$};
\draw (3.5,0)--(4.5,0);
\fill (4.5,0) circle (4/3pt) node[below=1pt] {$a_{m+1}^+$};

\draw (4.5,0)--(5.5,0);
\fill (5.5,0) circle (4/3pt) node[below=1pt] {$a_{m+2}^+$};
\draw (5.5,0)--(5.83,0);
\draw (6.3,0) node {{\bf\ldots}};
\draw (6.66,0)--(7,0);
\fill (7,0) circle (4/3pt) node[below=1pt] {$a_{m+k}^+$};
\draw (7,0)--(8,0);
\fill (8,0) circle (4/3pt) node[below=1pt] {$a_{m+k+1}^+$};

\end{scope}
\end{tikzpicture}

\caption{The ascending and the descending links for the Morse function $f\colon K_{m,k}\to\R/\Z$.}
\label{fig:ascdesc}
\end{figure}

From Figure~\ref{fig:ascdesc} we observe that the ascending and the descending links of the $0$--cell of $K_{m,k}$ are indeed trees. By the definition of $f$, each $2$--cell of $K_{m,k}$ contributes its diagonal loop to $f^{-1}(\text{$0$--cell})$. Furthermore, $f^{-1}(\text{$0$--cell})$ is a bouquet of these diagonal loops. Hence $f^{-1}(\text{$0$--cell})$ is a graph having a single $0$--cell and $(m+k)$ $1$--cells which are denoted in the Figure~\ref{fig:morse} by $A_i$, $B_j$.
\end{proof}

The above Proposition implies that $G_{m,k}\cong F_{m+k}\rtimes_{\phi_{m,k}}\Z$ for some monodromy automorphism $\phi_{m,k}$. We shall determine explicitly the automorphism $\phi_{m,k}$ for a particular choice of basis for $F_{m+k}$. Let $A_i$, $B_j$ be the diagonals of the $2$--cells of $K_{m,k}$, as shown in Figure~\ref{fig:morse}. Note that they have the following expressions in the generators of $G_{m,k}$:
\[
A_i=a_{i+1}^{-1}a_i,\quad 1\le i\le m;\qquad B_j=a_{m+j+1}^{-1}a_j,\quad 1\le j\le k.
\]
\begin{prop}\label{prop:phi}
For $0\le k\le m$, $G_{m,k}$ has the following explicit free-by-cyclic structure: 
\[
G_{m,k}\cong F_{m+k}\rtimes_{\phi_{m,k}}\Z
\] 
where
\[
F_{m+k}=\la A_1,\dots,A_m,B_1,\dots,B_{k}\ra;\qquad \Z=\la a_1\ra
\]
and the monodromy automorphism $\phi_{m,k}$ acts as follows (here overbar denotes the inverse):
\begin{align*}
\quad\phi_{m,k}\colon\quad & A_1\longmapsto A_1\\
& A_2\longmapsto A_1\,(A_2)\,\ov{A}_1\\
& A_3\longmapsto A_1A_2\,(A_3)\,\ov{A}_2\ov{A}_1\\
&\dots\\
& A_{m}\longmapsto A_1A_2\ldots A_{m-1}\,(A_m)\,\ov{A}_{m-1}\ldots \ov{A}_1\\[1em]
&B_1\longmapsto A_1A_2\ldots A_m\,(B_1)\\
&B_2\longmapsto A_1A_2\ldots A_m\,(B_1B_2)\,\ov{A}_1\\
&B_3\longmapsto A_1A_2\ldots A_m\,(B_1B_2B_3)\,\ov{A}_2\ov{A}_1\\
&\dots\\
&B_{k}\longmapsto A_1A_2\ldots A_m\,(B_1B_2\ldots B_{k})\,\ov{A}_{k-1}\ov{A}_{k-2}\ldots\ov{A}_2\ov{A}_1.
\end{align*}
Furthermore, $\phi_{m,k}$ is the restriction of $\phi_{m,m}$ to $F_{m+k}$.
\end{prop}

\begin{proof}
In the proof of Proposition~\ref{prop:morse} it was shown that $F_{m+k}$ is freely generated by all elements $A_i$, $B_j$.

As a generator of the $\Z$ factor we are free to choose any element that maps to a generator of $\pi_1(S^1)$; without loss of generality, we may take $\Z=\la a_1\ra$.

To get the action of the monodromy automorphism $\phi$ on the generators $A_i$, $B_j$ of $F_{m+k}$ we need to compute the conjugations $a_1A_ia_1^{-1}$ and $a_1B_ja_1^{-1}$. That is, we need to find words in generators $A_i$, $B_j$ which are equal to $a_1A_ia_1^{-1}$ and $a_1B_ja_1^{-1}$ in $K_{m,k}$.

\begin{figure}[h!]
\begin{tikzpicture}[scale=1.0]
\fill (3,3) circle (4/3pt);
\fill (7,3) circle (4/3pt);

\fill (0,0) circle (4/3pt);
\fill (1,0) circle (4/3pt);
\fill (2,0) circle (4/3pt);
\fill (3,0) circle (4/3pt);
\fill (4,0) circle (4/3pt);
\fill (5,0) circle (4/3pt);
\fill (6,0) circle (4/3pt);
\fill (7,0) circle (4/3pt);
\fill (8,0) circle (4/3pt);
\fill (9,0) circle (4/3pt);
\begin{scope}[thick]
\draw[->-=0.53] (0,0)--node[left=2pt] {$a_{1}$}(3,3);
\draw[->-=0.53] (5,0)--node[right=2pt] {$a_{i+1}$}(3,3);
\draw[->-=0.53] (5,0)--node[left=2pt] {$a_{i}$}(7,3);
\draw[->-=0.53] (3,3)--node[above=1pt] {$A_i$}(7,3);
\draw[->-=0.53] (9,0)--node[right=2pt] {$a_{1}$}(7,3);

\draw[->-=0.6] (0,0)--node[below=1pt] {$A_{1}$}(1,0);
\draw[->-=0.6] (1,0)--node[below=1pt] {$A_{2}$}(2,0);
\draw (2.55,0) node {{\bf\ldots}};
\draw[->-=0.6] (3,0)--node[below=1pt] {$A_{i-1}$}(4,0);
\draw[->-=0.6] (4,0)--node[below=1pt] {$A_{i}$}(5,0);
\draw[->-=0.6] (6,0)--node[below=1pt] {$A_{i-1}$}(5,0);
\draw (6.55,0) node {{\bf\ldots}};
\draw[->-=0.6] (8,0)--node[below=1pt] {$A_{2}$}(7,0);
\draw[->-=0.6] (9,0)--node[below=1pt] {$A_{1}$}(8,0);

\end{scope}

\draw[->-=0.53] (1,0)--(3,3);\draw (1.35,0.9) node {$a_{2}$};
\draw[->-=0.53] (2,0)--(3,3);\draw (2.05,0.9) node {$a_{3}$};
\draw[->-=0.53] (3,0)--(3,3);\draw (3.35,0.9) node {$a_{i-1}$};
\draw[->-=0.53] (4,0)--(3,3);\draw (3.9,1.0) node {$a_{i}$};

\draw[->-=0.53] (6,0)--(7,3);\draw (5.9,0.7) node {$a_{i-1}$};
\draw[->-=0.53] (7,0)--(7,3);\draw (7.25,0.7) node {$a_{3}$};
\draw[->-=0.53] (8,0)--(7,3);\draw (8.0,0.7) node {$a_{2}$};

\end{tikzpicture}
\caption{\label{fig:ai} The action of the monodromy automorphism on $A_i$.}
\end{figure}
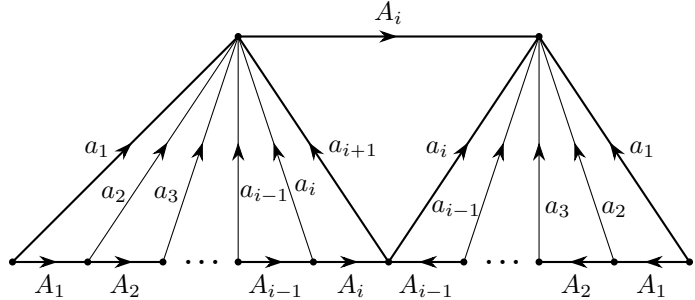

For $a_1A_ia_1^{-1}$, we start with the triangle having $A_i$ on top and $1$--cells $a_{i+1}$, $a_i$ forming two bottom sides. We would like to express $a_1a_{i+1}^{-1}$ and $a_ia_1^{-1}$ as products of free generators $A_i$, $B_j$. Since the descending link of the $0$--cell in $K_{m,k}$ is a tree, there exists a unique path in it connecting $a_1^+$ to $a_{i+1}^+$ and a unique path connecting $a_i^+$ to $a_1^+$. These paths correspond to paths $A_1A_2\dots A_{i-1}A_i$ and $\ov{A}_{i-1}\dots \ov{A}_2\ov{A}_1$, respectively, see Figure~\ref{fig:ai}. Thus, 
\[
a_1A_ia_1^{-1}= A_1A_2\ldots A_{i-1}\,(A_i)\,\ov{A}_{i-1}\ldots\ov{A}_2 \ov{A}_1.
\]

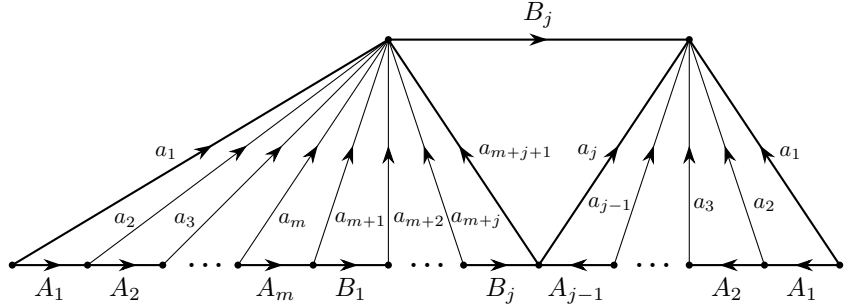
\begin{figure}[h!]
\begin{tikzpicture}[scale=1.0]
\fill (3,3) circle (4/3pt);
\fill (7,3) circle (4/3pt);

\fill (-2,0) circle (4/3pt);
\fill (-1,0) circle (4/3pt);
\fill (0,0) circle (4/3pt);
\fill (1,0) circle (4/3pt);
\fill (2,0) circle (4/3pt);
\fill (3,0) circle (4/3pt);
\fill (4,0) circle (4/3pt);
\fill (5,0) circle (4/3pt);
\fill (6,0) circle (4/3pt);
\fill (7,0) circle (4/3pt);
\fill (8,0) circle (4/3pt);
\fill (9,0) circle (4/3pt);
\begin{scope}[thick]
\draw[->-=0.53] (-2,0)--node[left=5pt] {\tinyb{$a_{1}$}}(3,3);
\draw[->-=0.53] (5,0)--node[right=2pt] {\tinyb{$a_{m+j+1}$}}(3,3);
\draw[->-=0.53] (5,0)--node[left=2pt] {\tinyb{$a_{j}$}}(7,3);
\draw[->-=0.53] (3,3)--node[above=1pt] {$B_j$}(7,3);
\draw[->-=0.53] (9,0)--node[right=2pt] {\tinyb{$a_{1}$}}(7,3);

\draw[->-=0.66] (-2,0)--node[below=1pt] {$A_{1}$}(-1,0);
\draw[->-=0.66] (-1,0)--node[below=1pt] {$A_{2}$}(0,0);
\draw (0.6,0) node {{\bf\ldots}};
\draw[->-=0.6] (1,0)--node[below=1pt] {$A_{m}$}(2,0);
\draw[->-=0.6] (2,0)--node[below=1pt] {$B_{1}$}(3,0);
\draw (3.55,0) node {{\bf\ldots}};
\draw[->-=0.6] (4,0)--node[below=1pt] {$B_{j}$}(5,0);
\draw[->-=0.6] (6,0)--node[below=1pt] {$A_{j-1}$}(5,0);
\draw (6.55,0) node {{\bf\ldots}};
\draw[->-=0.6] (8,0)--node[below=1pt] {$A_{2}$}(7,0);
\draw[->-=0.6] (9,0)--node[below=1pt] {$A_{1}$}(8,0);

\end{scope}

\draw[->-=0.53] (-1,0)--(3,3);\draw (-0.5,0.6) node {\tinyb{$a_{2}$}};
\draw[->-=0.53] (0,0)--(3,3);\draw (0.3,0.6) node {\tinyb{$a_{3}$}};
\draw[->-=0.53] (1,0)--(3,3);\draw (1.73,0.6) node {\tinyb{$a_{m}$}};
\draw[->-=0.53] (2,0)--(3,3);\draw (2.63,0.6) node {\tinyb{$a_{m+1}$}};
\draw[->-=0.53] (3,0)--(3,3);\draw (3.41,0.6) node {\tinyb{$a_{m+2}$}};
\draw[->-=0.53] (4,0)--(3,3);\draw (4.17,0.6) node {\tinyb{$a_{m+j}$}};

\draw[->-=0.53] (6,0)--(7,3);\draw (5.95,0.8) node {\tinyb{$a_{j-1}$}};
\draw[->-=0.53] (7,0)--(7,3);\draw (7.2,0.8) node {\tinyb{$a_{3}$}};
\draw[->-=0.53] (8,0)--(7,3);\draw (8.0,0.8) node {\tinyb{$a_{2}$}};

\end{tikzpicture}
\caption{\label{fig:bj} The action of the monodromy automorphism on $B_j$.}
\end{figure}

Similarly, for $a_1B_ja_1^{-1}$, we start with the triangle having $B_j$ on top and $a_{m+j+1}$, $a_j$ forming two bottom sides. Again, there are unique paths in the descending link from $a_1^+$ to $a_{m+j+1}^+$ and from $a_j^+$ to $a_1^+$. They correspond to words $A_1A_2\dots A_mB_1\dots B_j$ and $\ov{A}_{j-1}\ldots\ov{A}_2\ov{A}_1$, respectively, see Figure~\ref{fig:bj}. Hence,
\[
a_1B_ja_1^{-1}=A_1A_2\ldots A_m\,(B_1\ldots B_{j})\,\ov{A}_{j-1}\ldots\ov{A}_2\ov{A}_1.\qedhere
\]
\end{proof}

\section{Constructing a special cover for $G_{m,m}$}

In this section we construct a certain permutation representation for a group $G_{m,m}$ and show that it defines a finite cover for its presentation $2$--complex $K_{m,m}$, which can be embedded in an $(2m+1)$--dimensional torus. This allows us to construct a finite special cover for $K_{m,m}$, for all \emph{even} values of $m$.

\subsection{The permutation representation} We now define a right transitive action of $G_{m,m}$ on a certain set $H_{2m+1}$ of cardinality $2^{2m+1}$. 

\emph{Action set.} For any $n=1,\dots,2m+1$ denote $H_k$ to be the set of all tuples of length $k$ consisting of $0$'s and $1$'s, i.e. $H_n=\prod_{i=1}^n\{0,1\}$. There are natural inclusions 
\[
H_n\hookrightarrow H_{n+1},\quad (x_1,\dots,x_n)\mapsto (x_1,\dots,x_n,0), 
\]
and we identify $H_n$ with its image in $H_{n+1}$ under these inclusions. Also denote $H_n^*$ a subset of $H_{n+1}$ consisting of all tuples with the last coordinate~$1$:
\[
H_n^*=\{(x_1,\dots,x_n,1)\}\subset H_{n+1}.
\]
With the above identifications, we have $H_{n+1}=H_n\sqcup H_n^*$ (disjoint union). 

To define a \emph{right action} of a group $G$ on a set $X$, it suffices to associate to each $g\in G$ a permutation $\pi(g)$ of $X$ such that 
\[
\pi(gh)=\pi(h)\pi(g)\quad\text{for all $g,h\in G$.}
\]
Equivalently, a right action of $G$ on $X$ is a homomorphism of the opposite group $G^\circ$ to $\Sym(X)$, the group of all permutations of $X$, where $G^\circ$ equals $G$ as a set, with the new operation $\circ$ defined as
\[
a\circ b:=ba.
\]
We adopt the latter approach and construct the homomorphism from $G_{m,m}^\circ$ to $\Sym(H_{2m+1})$.

Recall that we have a natural tower of inclusions 
\[
G_{m,0}\subset G_{m,1}\subset \dots \subset  G_{m,m}.
\]
Since there are also inclusions
\[
H_{m+1}\subset H_{m+2}\subset \dots \subset H_{(m+1)+m}=H_{2m+1}
\]
this allows us to define the homomorphism $\pi\colon G_{m,m}^\circ\to \Sym(H_{2m+1})$ inductively by repeatedly extending the homomorphisms $G_{m,k-1}^\circ\to \Sym(H_{m+k})$ to $G_{m,k}^\circ\to \Sym(H_{m+k+1})$ for $k=1,\dots,m$ as follows.

\emph{Base of induction.} Let the $m+1$ generators $a_1$, \dots, $a_{m+1}$ of $G_{m,0}$ act on $H_{m+1}$ as flips in the respective coordinates, i.e. for $i=1,\dots,m+1$, set
\[\tag{$A_0$}
\pi(a_i)|_{H_{m+1}}:=\beta_i
\]
where $\beta_i\colon H_{2m+1}\to H_{2m+1}$ given by
\[
\beta_i(x_1,\dots,x_{i-1},x_i,x_{i+1},\dots)=(x_1,\dots,x_{i-1},1-x_i,x_{i+1},\dots)
\]
is the operator that changes the $i$-th coordinate from $0$ to $1$ and vice versa, fixing all others.

All the relations in $G_{m,0}$ (and $G_{m,0}^\circ$) are commutators $[a_i,a_{i+1}]=1$, $i=1,\dots,m$. Clearly, they are satisfied in $\Sym(H_{m+1})$ since operators $\beta_i$ pairwise commute. Thus we have a well-defined homomorphism $\pi\colon G_{m,0}^\circ\to \Sym(H_{m+1})$.

\emph{Inductive step}. For a fixed $k\in\{1,\dots,m\}$, suppose that $\pi\colon G_{m,k-1}^\circ\to\Sym(H_{m+k})$ is already defined. In particular, this implies that $H_{m+k}$ is invariant under $\pi(a_j)$ for all $j=1,\dots,m+k$. Also suppose that the following property holds:
\[\tag{$P_k$}
\text{for all $j=k,\dots,m$,\quad} \pi(a_j)|_{H_{m+k}}=\beta_j|_{H_{m+k}}.
\]
The base of induction above guarantees that these suppositions are true for $k=1$. Our goal is to extend the homomorphism $\pi|_{G_{m,k-1}^\circ}$ to $\pi\colon G_{m,k}^\circ\to\Sym(H_{m+k+1})$. Since $H_{m+k+1}=H_{m+k}\sqcup H_{m+k}^*$, it will suffice to define $\pi(a_j)|_{H_{m+k}^*}$ for $j=1,\dots,m+k$, and $\pi(a_{m+k+1})|_{H_{m+k}\sqcup H_{m+k}^*}$.

To this end, we set
\[\tag{$A_1$}
\pi(a_{m+k+1})|_{H_{m+k}\sqcup H_{m+k}^*}:=\beta_{m+k+1}
\]
and for all $1\leq j\leq m+k$,
\[\tag{$A_2$}
\pi(a_j)|_{H_{m+k}^*}:=\beta_{m+k+1}\cdot \varphi_{k,m+k}\cdot\pi(a_j)|_{H_{m+k}}\cdot \varphi_{k,m+k}\cdot \beta_{m+k+1},
\]
where $\cdot$ denotes the composition and $\varphi_{k,m+k}\colon H_{m+k+1}\to H_{m+k+1}$ is the involution that interchanges $k$-th and $(m+k)$-th coordinates leaving all other coordinates fixed:
\[
\varphi_{k,m+k}\colon (x_1,\dots,x_k,\dots,x_{m+k},x_{m+k+1})\mapsto (x_1,\dots,x_{m+k},\dots,x_k,x_{m+k+1}).
\]

In other words, we transfer the action of $G_{m,k-1}$ from $H_{m+k}$ to $H_{m+k}^*$ while twisting it with $\varphi_{k,m+k}$. Notice that, with the above definitions, both sets $H_{m+k}$ and $H_{m+k}^*$ are invariant under $\pi(a_j)$ for $j=1,\dots,m+k$. All the relations involving generators $a_1,\dots,a_{m+k}$ are satisfied on $H_{m+k+1}=H_{m+k}\sqcup H_{m+k}^*$ since  they hold true on $H_{m+k}$ and the conjugation by $\beta_{m+k+1}\cdot\nobreak\varphi_{k,m+k}$ is a homomorphism between permutation groups on $H_{m+k}$ and $H_{m+k}^*$.

The only relation in $G_{m,k}$ involving the last generator $a_{m+k+1}$ is 
\[
a_{m+k+1}^{-1}a_ka_{m+k+1}=a_{m+k},
\]
which translates to
\[
a_{m+k+1}\circ a_k\circ a_{m+k+1}^{-1}=a_{m+k}
\]
in $G_{m,k}^\circ$.

Since $\beta_{m+k+1}$ sends $H_{m+k}$ to $H_{m+k}^*$, the left-hand side of this relation acts on $H_{m+k}$ as follows:
\begin{align*}
&\pi(a_{m+k+1})\cdot\pi(a_k)\cdot\pi(a_{m+k+1}^{-1})|_{H_{m+k}}\\
&=\pi(a_{m+k+1})\cdot\pi(a_k)\cdot \beta_{m+k+1}|_{H_{m+k}}\\
&=\pi(a_{m+k+1})\cdot\pi(a_k)|_{H_{m+k}^*}\cdot \beta_{m+k+1}|_{H_{m+k}}\\
&=\pi(a_{m+k+1})\cdot (\beta_{m+k+1}\cdot \varphi_{k,m+k}\cdot\pi(a_k)|_{H_{m+k}}
\cdot \varphi_{k,m+k}\cdot \beta_{m+k+1})\cdot \beta_{m+k+1}|_{H_{m+k}}\\
&=\pi(a_{m+k+1})\cdot \beta_{m+k+1}\cdot \varphi_{k,m+k}\cdot\pi(a_k)|_{H_{m+k}}\cdot \varphi_{k,m+k}|_{H_{m+k}}\\
&=(\pi(a_{m+k+1})\cdot \beta_{m+k+1})\cdot \varphi_{k,m+k}\cdot\pi(a_k)|_{H_{m+k}}\cdot \varphi_{k,m+k}|_{H_{m+k}}\\
&=\id\cdot\,\varphi_{k,m+k}\cdot\pi(a_k)|_{H_{m+k}}\cdot \varphi_{k,m+k}|_{H_{m+k}}=\text{[by~($P_k$)]}\\
&=\varphi_{k,m+k}\cdot \beta_k|_{H_{m+k}}\cdot \varphi_{k,m+k}|_{H_{m+k}}=\beta_{m+k}|_{H_{m+k}}=\pi(a_{m+k})|_{H_{m+k}},
\end{align*}
where the last equality holds due to the inductive definition of $\pi$. Thus, the both sides of the above relation act the same on $H_{m+k}$.

Analogously, on $H_{m+k}^*$, the left-hand side acts as:
\begin{multline*}
\pi(a_{m+k+1})\cdot\pi(a_k)\cdot\pi(a_{m+k+1}^{-1})|_{H_{m+k}^*}
=\beta_{m+k+1}\cdot\pi(a_k)|_{H_{m+k}}\cdot \beta_{m+k+1}|_{H_{m+k}^*}\\
=\text{[by ($P_k$)]}=\beta_{m+k+1}\cdot \beta_k|_{H_{m+k}}\cdot \beta_{m+k+1}|_{H_{m+k}^*}=\beta_k|_{H_{m+k}^*},
\end{multline*}
and the right-hand side:
\begin{multline*}
\pi(a_{m+k})|_{H_{m+k}^*}=\beta_{m+k+1}\cdot \varphi_{k,m+k}\cdot\pi(a_{m+k})|_{H_{m+k}}\cdot \varphi_{k,m+k}\cdot \beta_{m+k+1}|_{H_{m+k}^*}\\
=\text{[by the inductive definition]}=\beta_{m+k+1}\cdot \varphi_{k,m+k}\cdot \beta_{m+k}\cdot \varphi_{k,m+k}\cdot \beta_{m+k+1}|_{H_{m+k}^*}\\
=\beta_{m+k+1}\cdot \beta_k\cdot \beta_{m+k+1}|_{H_{m+k}^*}=\beta_k|_{H_{m+k}^*}.
\end{multline*}
Since the two sides act the same on $H_{m+k}\sqcup H_{m+k}^*=H_{m+k+1}$, the above relation is satisfied in $\Sym(H_{m+k+1})$, which proves that $\pi$ is well-defined on $G_{m,k}^\circ$.

It remains to be proved that the auxiliary condition ($P_k$) is preserved under the inductive step, i.e. that ($P_k$) implies ($P_{k+1}$). 

Indeed, ($P_k$) means that $\pi(a_j)|_{H_{m+k}}=\beta_j|_{H_{m+k}}$ for $j=k,\dots,m$. Thus, for $j>k$,
\begin{multline*}
\pi(a_j)|_{H_{m+k}^*}=\beta_{m+k+1}\cdot \varphi_{k,m+k}\cdot\pi(a_j)|_{H_{m+k}}\cdot \varphi_{k,m+k}\cdot \beta_{m+k+1}|_{H_{m+k}^*}\\
=\beta_{m+k+1}\cdot\varphi_{k,m+k}\cdot \beta_j|_{H_{m+k}}\cdot \varphi_{k,m+k}\cdot \beta_{m+k+1}|_{H_{m+k}^*}\\
=\text{[since $j\ne k,\,m+k$]}=\beta_j|_{H_{m+k}^*}.
\end{multline*}
This proves that $\pi(a_j)|_{H_{m+k+1}}=\beta_j|_{H_{m+k+1}}$ for $j=k+1,\dots,m$, i.e. that ($P_{k+1}$) holds.

This finishes the inductive construction of the homomorphism 
\[
\pi\colon G_{m,m}^\circ\to \Sym(H_{2m+1})
\] 
and the proof that it is well-defined. Thus one gets a right action of $G_{m,m}$ on $H_{2m+1}$ which will also be denoted $\pi$.

Figure~\ref{fig:1} shows the permutation representation for $G_{2,2}$.


\newcommand{\hhh}{
\begin{picture}(0,0)
\put(1.6,0){\line(1,0){66.8}}\put(0,0){\circle{3}}
\put(0,1.6){\line(0,1){66.8}}
\put(1.3,1.3){\line(1,1){17.5}}
\put(21.6,20){\line(1,0){26.8}}\put(20,20){\circle{3}}
\put(20,21.6){\line(0,1){26.8}}
\put(21.6,50){\line(1,0){26.8}}\put(20,50){\circle{3}}
\put(50,21.6){\line(0,1){26.8}}\put(50,20){\circle{3}}
\put(1.6,70){\line(1,0){66.8}}\put(0,70){\circle{3}}
\put(70,1.6){\line(0,1){66.8}}\put(70,0){\circle{3}}
\put(1.3,68.7){\line(1,-1){17.5}}
\put(51.3,18.7){\line(1,-1){17.5}}\put(50,50){\circle{3}}
\put(51.3,51.3){\line(1,1){17.5}}\put(70,70){\circle{3}}
\end{picture}
}

\begin{figure}
\begin{center}
\unitlength=1.5pt
\begin{picture}(220,230)(0,-10)
\put(42.6,160){\circle*{3}}
\put(0,20){\hhh}
\put(120,0){\hhh}
\put(20,140){\hhh}
\put(140,120){\hhh}
\multiput(4.5,19.7)(70,0){2}{\line(6,-1){116.5}}
\multiput(4.5,89.7)(70,0){2}{\line(6,-1){116.5}}
\multiput(24.5,39.7)(30,0){2}{\line(6,-1){116.5}}
\multiput(24.5,69.7)(30,0){2}{\line(6,-1){116.5}}
\multiput(24.5,139.7)(70,0){2}{\line(6,-1){116.5}}
\multiput(24.5,209.7)(70,0){2}{\line(6,-1){116.5}}
\multiput(44.5,159.7)(30,0){2}{\line(6,-1){116.5}}
\multiput(44.5,189.7)(30,0){2}{\line(6,-1){116.5}}
\multiput(3,21.7)(70,0){2}{\line(1,6){19.4}}
\multiput(3,91.7)(70,0){2}{\line(1,6){19.4}}
\multiput(23,41.7)(30,0){2}{\line(1,6){19.4}}
\multiput(23,71.7)(30,0){2}{\line(1,6){19.4}}
\multiput(123,1.7)(70,0){2}{\line(1,6){19.4}}
\multiput(123,71.7)(70,0){2}{\line(1,6){19.4}}
\multiput(143,21.7)(30,0){2}{\line(1,6){19.4}}
\multiput(143,51.7)(30,0){2}{\line(1,6){19.4}}
\put(44,154){\tiny00000}
\put(74,161){\tiny10000}
\put(44,184){\tiny01000}
\put(76,191){\tiny11000}
\put(8.6,142){\tiny00100}
\put(95,141){\tiny10100}
\put(25,212){\tiny01100}
\put(95,211){\tiny11100}
\put(164,135){\tiny00010}
\put(194,141){\tiny10010}
\put(164,165){\tiny01010}
\put(195,169){\tiny11010}
\put(142.8,114){\tiny00110}
\put(215,121){\tiny10110}
\put(145,191){\tiny01110}
\put(215,191){\tiny11110}
\put(22,34){\tiny00001}
\put(54,41){\tiny10001}
\put(24,64){\tiny01001}
\put(55.5,70){\tiny11001}
\put(-13,22){\tiny00101}
\put(74,21){\tiny10101}
\put(-13,92){\tiny01101}
\put(74,92){\tiny11101}
\put(142,14){\tiny00011}
\put(174,21){\tiny10011}
\put(144,45){\tiny01011}
\put(175,49){\tiny11011}
\put(122.8,-6){\tiny00111}
\put(195,1){\tiny10111}
\put(123.5,72){\tiny01111}
\put(194,72){\tiny11111}
\put(54,162){\it1}\put(54,142){\it1}\put(54,192){\it1}\put(54,212){\it1}
\put(44,172){\it2}\put(24,172){\it2}\put(74,172){\it2}\put(94,166){\it2}
\put(29,151){\it3}\put(29,193){\it3}\put(74,195){\it3}\put(76,147){\it3}
\put(174,142){\it3}\put(170,113){\it3}\put(174,172){\it3}\put(178,183){\it3}
\put(164,152){\it2}\put(144,156){\it2}\put(194,152){\it2}\put(214,152){\it2}
\put(149,131){\it1}\put(155,179){\it1}\put(196,177){\it1}\put(196,127){\it1}
\put(127,207){\it4}\put(117,197){\it4}\put(118,185){\it4}\put(118,170){\it4}
\put(118,155){\it4}\put(118,140){\it4}\put(116,129){\it4}\put(109,118){\it4}
\put(36,42){\it1}\put(35,22){\it1}\put(38,72){\it3}\put(40,92){\it3}
\put(17,52){\it4}\put(-3,54){\it4}\put(47,52){\it4}\put(67,50){\it4}
\put(9,31){\it3}\put(6,77){\it1}\put(63,84){\it1}\put(56,27){\it3}
\put(155,22){\it1}\put(154,-7){\it1}\put(157,51){\it3}\put(162,64){\it3}
\put(137,34){\it4}\put(117,39){\it4}\put(167,33){\it4}\put(187,34){\it4}
\put(130,12){\it3}\put(135,59){\it1}\put(180,61){\it1}\put(176,7){\it3}
\put(104,87){\it2}\put(100,76){\it2}\put(98,65){\it2}\put(98,50){\it2}
\put(98,35){\it2}\put(98,20){\it2}\put(96,9){\it2}\put(86,-2){\it2}
\put(2,117){\it5}\put(14,115){\it5}\put(26,113){\it5}\put(36,112){\it5}
\put(55,111){\it5}\put(65,110){\it5}\put(77,108){\it5}\put(89,105){\it5}
\put(122,98){\it5}\put(134,95){\it5}\put(146,93){\it5}\put(156,92){\it5}
\put(175,91){\it5}\put(185,90){\it5}\put(197,88){\it5}\put(209,86){\it5}
\end{picture}
\unitlength=1pt
\end{center}
\caption{The case of $m=2$, $k=2$: the action of $G_{2,2}=\la a_1,a_2,a_3,a_4,a_5\mid [a_1,a_2]=1, [a_2,a_3]=1, a_4^{-1}a_1a_4=a_3, a_5^{-1}a_2a_5=a_4\ra$ on $H_5$. Elements of $H_5$ are arranged at vertices of the hypercube graph marked with the corresponding tuples of 0,1's. (Thus, each edge of this graph corresponds to a pair of opposite edges in the $1$--skeleton of the $5$--dimensional torus $\T_5$ defined below.) The subset $H_3$ is represented by the upper left-hand corner subgraph, and $H_4$ by the upper half of the picture. 
If $\pi(a_i)$ interchanges vertices $u$ and $v$ we mark the edge $uv$ with the italicized digit $i$. }
\label{fig:1}
\end{figure}
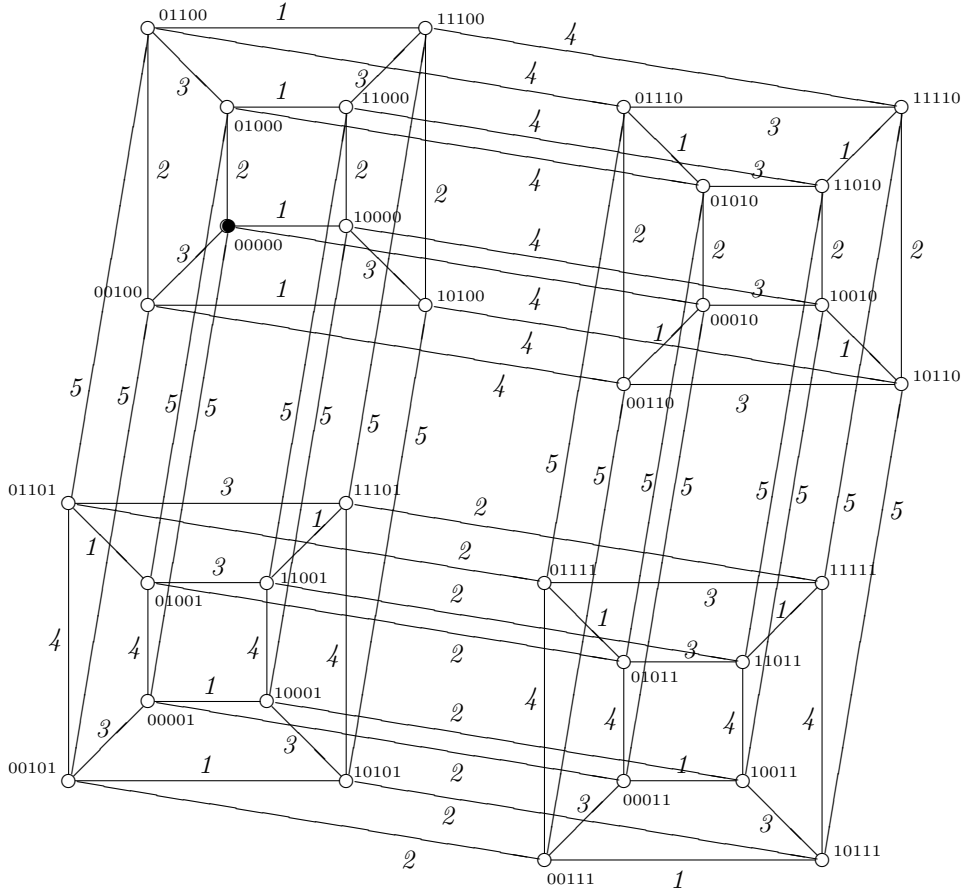

\begin{prop}\label{prop:c1}
The right action $\pi$ of $G_{m,m}$ on $H_{2m+1}$, defined above, has the following properties:
\begin{enumerate}
\item $G_{m,k}$ acts transitively on $H_{m+k+1}$ for all $k=0,\dots,m$.
\item Each generator $a_i$, $i=1,\dots,2m+1$, acts as an involution on $H_{2m+1}$.
\item For any $v\in H_{2m+1}$, and any $a_i$, $i=1,\dots,2m+1$, $\pi(a_i)v$ differs from $v$ in exactly one coordinate. In particular, $\pi(a_i)$ has no fixpoints.
\item For any $i\ne j$, $\pi(a_ia_j)$ has no fixpoints.
\end{enumerate}
\end{prop}
\begin{proof}
(1) An easy induction. The case $k=1$ is obvious since $G_{m,0}$ acts on $H_{m+1}$ by coordinate flips. So one can start with any $(m+1)$--tuple of 0,1's and obtain any other $(m+1)$--tuple by changing one coordinate at a time. Suppose now that $G_{m,k-1}$ acts transitively on $H_{m+k}$. Then by~($A_2$), $H^*_{m+k}$ comprises another orbit for $G_{m,k-1}$ and $a_{m+k+1}$ glues $H_{m+k}$ and $H^*_{m+k}$ into one orbit for $G_{m,k}$ by~($A_1$) thus proving that $G_{m,k}$ is transitive on $H_{m+k+1}=H_{m+k}\sqcup H_{m+k}^*$.

(2),(3) Follow by induction from formulas ($A_0$)--($A_2$).

(4) Again, this is obvious for $a_i$, $a_j$ with $1\le i,j\le m+1$ acting on $H_{m+1}$ since they act as different coordinate flips $\beta_i$, $\beta_j$. Suppose that the statement is proven for some $k\in\{1,\dots,m+1\}$, for all $a_i$, $a_j$, $1\le i,j\le m+k$ acting on $H_{m+k}$. Then $\pi(a_ia_j)$ has no fixpoints on $H^*_{m+k}$ either, since otherwise if $v\in H^*_{m+k}$ is such a fixpoint, then by~($A_2$), $\varphi_{k,m+k}\cdot\beta_{m+k+1}(v)$ would be a fixpoint for $\pi(a_ia_j)$ in $H_{m+k}$. Finally, if, say, $i=m+k+1$ then $\pi(a_i)$ changes the last, $(m+k+1)$-st, coordinate on $H_{m+k+1}$, whereas for $j<i$, $\pi(a_j)$ preserves both subsets $H_{m+k}$ and $H_{m+k}^*$, so it doesn't change the $(m+k+1)$-st coordinate. Hence, the composition of $\pi(a_j)$ and $\pi(a_i)$ has no fixpoints.
\end{proof}

\subsection{A $(2m+1)$--torus cover} 

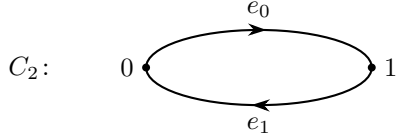
\begin{figure}[!h] 
\begin{tikzpicture}[scale=1.0]
\begin{scope}[thick]
\fill (0,0) circle (1.5pt) node[left=1pt] {$0$};
\fill (3,0) circle (1.5pt) node[right=1pt] {$1$};
\draw[->-=0.53] (0,0) arc (180:0:1.5cm and 0.5cm); 
\draw (1.5,0.5) node[above=1pt] {$e_0$};
\draw[->-=0.53] (3,0) arc (360:180:1.5cm and 0.5cm);
\draw (1.5,-0.5) node[below=1pt] {$e_1$};
\draw (-1,0) node[left=1pt] {$C_2\colon$};
\end{scope}
\end{tikzpicture}
\caption{The CW complex $C_2$.\label{fig:c2}}
\end{figure}

Let $C_2=\R/2\Z$ be a $1$--dimensional CW complex with the following CW structure: its $0$--cells are $0 + {2\Z}$ and $1+ {2\Z}$, which we denote by $0$ and $1$ respectively. The $1$--cells are $[0,1] + {2\Z}$ and $[1,2]+{2\Z}$, which we denote by $e_0$ and $e_1$ respectively, see Figure~\ref{fig:c2}. 

We denote by $\T_n$ the CW complex $\R^n/(2\Z)^n\cong(\R/2\Z)^n$ with the product CW structure. Notice that the natural action of $(2\Z)^n$ on $\R^n$ preserves the standard unit cubulation of $\R^n$, hence induces the structure of a cubical complex on $\T_n$. Observe that 
$\T_n$ is homeomorphic to an $n$--dimensional torus.

In what follows, it will be convenient to parametrize points of $\T_n$ by $n$--tuples $(x_1,\dots,x_n)$ of numbers from $[0,2]$ viewed up to the identification $0\sim2$.

The $0$--skeleton of $\T_n$ is naturally identified with the set $H_n$ of all $n$-tuples of $\{0,1\}$ introduced before.

The $1$--cells of $\T_n$ are formed by fixing an edge $e_0$ or $e_1$ in some factor of $\T_n=C_2\times C_2\times\dots \times C_2$, say, in position $i$, and taking product with vertices $0$ or $1$ in all other positions. 
(So if two $0$--cells of $\T_0$ differ in only one coordinate, then there is a unique directed edge in $\T_n^{(1)}$ from the first $0$--cell to the second one and a unique directed edge from the second one to the first one.) Thus, a typical $1$--cell in $\T_n$ can be identified with a product of the form 
\[
v_1\times v_2\times \dots\times v_{i-1}\times e_\alpha\times v_{i+1}\times\dots\times v_n,
\] 
where each $v_j\in\{0,1\}$ and $\alpha=0$ or $1$.

Similarly, an arbitrary $2$--cell of $\T_{n}$ is a product 
\[
v_1\times \dots\times v_{i-1}\times e_\alpha\times v_{i+1}\times\dots\times v_{j-1}\times e_\beta\times v_{j+1}\times\dots\times v_{n}
\] 
for some choice of $1\le i, j\le n$ ($i\ne j$), with each $v_k\in\{0,1\}$, and $\alpha,\beta\in\{0,1\}$.

Let $K_{m,m}$ be the presentation $2$--complex for $G_{m,m}$. Recall that it consists of one $0$--cell, $(2m+1)$ $1$--cells corresponding to the generators $a_1,\dots,a_{2m+1}$ of $G_{m,m}$ and $2m$ $2$--cells corresponding to the relations of $G_{m,m}$. 

Consider the right action $\pi\colon G_{m,m}^\circ\to\Sym(H_{2m+1})$ constructed in the previous section and denote
\[
S=\{g\in G_{m,m}\mid \pi(g)(0,0,\dots,0)=(0,0,\dots,0)\}
\]
the stabilizer of the point $(0,0,\dots,0)$ in $G_{m,m}$. Subgroup $S$ defines a finite covering $\widehat K_m\to K_{m,m}$ whose properties we now describe.

\begin{prop}\label{prop:emb}
The covering space $\widehat K_m$ cellularly embeds into the $2$--skeleton of $\T_{2m+1}$.
\end{prop}
\begin{proof} 
The $0$--cells of $\widehat K_m$ are in one-to-one correspondence with the right cosets $S\backslash G_{m,m}$. Since the action of $G_{m,m}$ is transitive on $H_{2m+1}$ by Proposition~\ref{prop:c1}(1), $\widehat K_m^{(0)}$ consists of ${|G_{m,m}:S|}=2^{2m+1}$ vertices which we can identify with $\T_{2m+1}^{(0)}$, the $0$--skeleton of $\T_{2m+1}$, which was earlier identified with the set $H_{2m+1}$ of all $(2m+1)$--tuples consisting of $\{0,1\}$.

The $1$--cells of $\widehat K_m$ are in one-to-one correspondence with pairs of right cosets $(Sg,Sga_i)$ where $a_i$, $i=1,\dots, 2m+1$ runs through all the generators of $G_{m,m}$. Proposition~\ref{prop:c1}(3) guarantees that each such $1$--cell is not a loop, and it actually belongs to the $1$--skeleton of the $(2m+1)$--torus $\T_{2m+1}$.

The $2$--cells of $\widehat K_m$ are lifts of the $2$--cells in $K_{m,m}$. Each such $2$--cell is uniquely determined by the base vertex (a lift of the base vertex of $K_{m,m}$) and by the fixed cyclic order of the relator word of $G_{m,m}$, which defines the attaching map of a $2$--cell in $K_{m,m}$. 
Indeed, since $\pi\colon G_{m,m}^\circ\to \Sym(H_{2m+1})$ is a homomorphism, every relator word $w$ of $G_{m,m}$ acts as the identical permutation. There are two types of relators in the presentation for $G_{m,m}$:
\begin{gather*}
a_ia_{i+1}a_i^{-1}a_{i+1}^{-1}=1\quad\text{for $i=1,\dots,m$,\quad and}\\
a_{m+j+1}^{-1}a_ja_{m+j+1}a_{m+j}^{-1}=1\quad\text{for $j=1,\dots,m$,}
\end{gather*}
each of which has length $4$. Thus they define length $4$ loops in $\widehat K_m$, based at every vertex, each of such loops has to be filled with a $2$--cell because these loops must be nullhomotopic when projected to $K_{m,m}$.
Thus, each $2$--cell of $\widehat K_m$ can be given by a word $a_ia_ja_l^{-1}a_j^{-1}$ for some values $i\ne j$, $j\ne l$, see Figure~\ref{fig:2cell}.

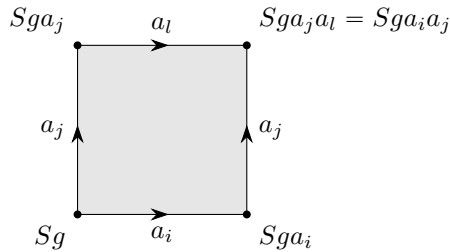
\begin{figure}[!h] 
\begin{tikzpicture}[scale=0.75]

\begin{scope}[color=black!10]
\fill (0,0)--(3,0)--(3,3)--(0,3);
\end{scope}

\fill (0,0) circle (2pt) node[below left=1pt] {$Sg$};
\fill (0,3) circle (2pt) node[above left=1pt] {$Sga_j$};
\fill (3,0) circle (2pt) node[below right=1pt] {$Sga_i$};
\fill (3,3) circle (2pt) node[above right=1pt] {$\lefteqn{Sga_ja_l=Sga_ia_j}$};

\begin{scope}[]
\draw[->-=0.53] (0,0)--node[left=1pt] {$a_{j}$}(0,3);
\draw[->-=0.54] (0,3)--node[above=1pt] {$a_l$}(3,3);
\draw[->-=0.54] (0,0)--node[below=1pt] {$a_i$} (3,0);
\draw[->-=0.53] (3,0)--node[right=1pt] {$a_{j}$}(3,3);
\end{scope}

\end{tikzpicture}
\caption{A typical $2$--cell in $\widehat K_m$.\label{fig:2cell}}
\end{figure}

If we identify cosets $S\backslash G_{m,m}$ with $\T_{2m+1}^{(0)}\equiv H_{2m+1}$, Proposition~\ref{prop:c1}(3) shows that for any $i=1,\dots,2m+1$, every edge $(Sg,Sga_i)$ changes only one coordinate of the $(2m+1)$--tuple of $\{0,1\}$ representing vertex $Sg$, therefore it maps to a suitable $1$--cell of $\T_{2m+1}$.

Let's show that each $2$--cell of the above form at a vertex $Sg\in \widehat K_m^{(0)}$ naturally embeds into $\T_{2m+1}^{(2)}$ under the embedding induced by the embeddings of $\widehat K_m^{(0)}$ and $\widehat K_m^{(1)}$ to $\T_{2m+1}^{(1)}$ introduced above.
Denote $p,q,r,s$ the positions in $\{1,\dots,2m+1\}$ in which the endpoints of the following edges differ: $(Sg,Sga_i)$, $(Sga_i, Sga_ia_j)$, $(Sg,Sga_j)$, $(Sga_j,Sga_ja_l)$, respectively. By Proposition~\ref{prop:c1}(4), $a_ia_j$ and $a_ja_l$ have no fixpoints on $H_{2m+1}$, hence $p\ne q$ and $r\ne s$. And since the square above is commutative, we conclude that $2$--element sets $\{p,q\}$ and $\{r,s\}$ are equal. Again, by Proposition~\ref{prop:c1}(2), $a_i$ and $a_j$ act as involutions, hence $p\ne r$, since otherwise $a_i^{-1}a_j=a_ia_j$ would have a fixpoint $Sga_i$. Therefore, $p=s$, $q=r$, and the $2$--cell under consideration actually belongs to $\T_{2m+1}^{(2)}$ since each pair of its parallel edges changes coordinates of vertices in the same position, one in position $p=s$ and another in $q=r$.
\end{proof}

\subsection{Exploring hyperplane pathologies} We will need the following description of hyperplanes and walls in the $n$--torus $\T_n$.

\begin{lem}\label{lem:hpw} The hyperplanes and oriented walls in $\T_n$ are in $1$--$1$ correspondence with pairs $(i,e_\alpha)$, where $1\le i\le n$ and $\alpha\in\{0,1\}$. More explicitly: 
\begin{itemize}
\item[(1)] the hyperplane corresponding to the pair $(i,e_\alpha)$ is a subset of $\T_n$ of one of the following two types:
\[
\big\{\,(x_1,\dots x_{i-1},\tfrac 12,x_{i+1},\dots,x_n)\mid x_j\in [0,2]\,\big\}
\]
if $e_\alpha=e_0$, and
\[
\big\{\,(x_1,\dots x_{i-1},\tfrac 32,x_{i+1},\dots,x_n)\mid x_j \in [0,2]\,\big\}
\]
if $e_\alpha=e_1$ (with the identification $0\sim2$).

\item[(2)] the oriented wall through a $1$--cell 
\[
v_1\times \dots\times v_{i-1}\times e_\alpha\times v_{i+1}\times\dots\times v_{n}
\] 
consists of all $1$--cells 
\[
u_1\times \dots\times u_{i-1}\times e_\alpha\times u_{i+1}\times\dots\times u_{n}
\] 
with $i$, $e_\alpha$ fixed, and $u_k$'s taking all possible values of $\{0,1\}$.
\end{itemize}
The oriented wall in \emph{(2)} is dual to the corresponding hyperplane in \emph{(1)}. 
\end{lem}
\begin{proof}
(1) Recall that the structure of a cube complex on $\T_n\cong\R^n/(2\Z)^n$ is induced by the standard cubulation of $\R^n$. Hyperplanes in $\R^n$ are subsets of the form
\[
H_{i,k}=\big\{(x_1,\dots,x_{i-1},k+\tfrac12,x_{i+1},\dots,x_n)\mid x_j\in\R\big\}, \quad 1\le i\le n,\quad k\in\Z.
\]
Modding out by the action of $(2\Z)^n$ yields the result.

(2) Recall that the oriented wall containing a $1$--cell $a$ of a cube complex $X$ is the class of all oriented $1$--cells of $X$ which are connected to $a$ through a sequence of elementary parallelisms via the $2$--cells of $X$. Notice that an arbitrary $1$--cell of $\T_{n}$ is a product of the form: $v_1\times \dots\times v_{i-1}\times e_\alpha\times v_{i+1}\times\dots\times v_{n}$, where each vertex $v_k\in\{0,1\}$ and $\alpha=0$ or $1$, and an arbitrary $2$--cell of $\T_{n}$ is a product 
\[
v_1\times \dots\times v_{i-1}\times e_\alpha\times v_{i+1}\times\dots\times v_{j-1}\times e_\beta\times v_{j+1}\times\dots\times v_{n}
\] 
for some choice of $1\le i, j\le n$ ($i\ne j$) and $\alpha,\beta\in\{0,1\}$. Thus, the elementary parallelism via the above $2$--cell establishes equivalence of the $1$--cells 
\[
v_1\times \dots\times v_{i-1}\times e_\alpha\times v_{i+1}\times\dots\times v_{j}\times\dots\times v_{n}
\] 
and
\[
v_1\times \dots\times v_{i-1}\times e_\alpha\times v_{i+1}\times\dots\times (1-v_{j})\times\dots\times v_{n}.
\] 
Since index $j$ varies independently of $i$, we conclude that any $1$--cell $u_1\times \dots\times u_{i-1}\times e_\alpha\times u_{i+1}\times\dots\times u_{n}$, $u_k\in\{0,1\}$, is contained in the parallelism class of $v_1\times \dots\times v_{i-1}\times e_\alpha\times v_{i+1}\times\dots\times v_{n}$.
\end{proof}

Now we show that the complex $\widehat K_m$ does not have three of the four pathologies in the definition of a special cube complex.

\begin{prop}\label{prop:clean} \
\begin{itemize} 
\item[(a)] Hyperplanes of $\widehat K_m$ do not self-intersect.
\item[(b)] Hyperplanes of $\widehat K_m$ do not self-osculate.
\item[(c)] Hyperplanes of $\widehat K_m$ are two-sided.
\end{itemize}
\end{prop}

\begin{proof}
It is convenient to work with the oriented walls dual to hyperplanes. Since, by Proposition~\ref{prop:emb}, $\widehat K_m$ is a square subcomplex of $\T_{2m+1}$, every wall of $\widehat K_m$ is a subset of some wall of $\T_{2m+1}$. By Lemma~\ref{lem:hpw}, the walls in $\T_{2m+1}$ consist of all $1$--cells of the form $u_1\times \dots\times u_{i-1}\times e_\alpha\times u_{i+1}\times\dots\times u_{2m+1}$ for a fixed $i$, $e_\alpha$, and arbitrary $u_k\in\{0,1\}$.

If a hyperplane of $\widehat K_m$ were self-intersecting, the corresponding wall would contain edges with $e_\alpha$ in two different coordinate positions $i$ and $j$, which is impossible. This proves (a).

If a hyperplane of $\widehat K_m$ were self-osculating, the corresponding wall would contain a pair of edges $u_1\times u_2\times \dots\times u_{i-1}\times e_\alpha\times u_{i+1}\times\dots\times u_{2m+1}$ and $v_1\times v_2\times \dots\times v_{i-1}\times e_\alpha\times v_{i+1}\times\dots\times v_{2m+1}$ with common extremities: either their origins or their termini coincide (for direct self-osculation), or the origin of one edge coincides with the terminus of the other (for indirect self-osculation). In either case the tuples $(u_1,\dots,u_{i-1},u_{i+1},\dots,u_{2m+1})$ and $(v_1,\dots,v_{i-1},v_{i+1},\dots,v_{2m+1})$ are equal, which means that the original $1$--cells are equal and there is actually no self-osculation happening. This proves (b).

To prove (c) we observe that a hyperplane $H$ of $\widehat K_m$ lies in a unique hyperplane in $\T_{2m+1}$. In particular, by the above lemma, in the coordinate system on $\T_{2m+1}$, the hyperplane $H$ has the following description:
\[
H=\big\{\,(x_1,\dots x_{i-1},t,x_{i+1},\dots,x_{2m+1})\,\big\}
\]
for some $1\le i\le 2m+1$, $t=\frac12$ or $\frac32$, and some values from $[0,2]$ for the rest of the variables. Since $\widehat K_m$ is a square complex, $H$ is the union of mid-cubes of some set of square $2$--cells. Hence each point $z$ of $H$ belongs to a square $2$--cell $C$ of $\widehat K_m$ of the form
\[
C=v_1\times \dots\times v_{i-1}\times e_\alpha\times v_{i+1}\times\dots\times v_{j-1}\times e_\beta\times v_{j+1}\times\dots\times v_{2m+1}
\]
for some $1\le j\le 2m+1$, where all $v_k$'s are $0$ or $1$. (The index $j$ may be less than or bigger than $i$.)

Suppose that $e_\alpha=e_0$ so that $t=\frac12$. Then $z$ actually has coordinates:
\[
z=(v_1,\dots v_{i-1},\tfrac12,v_{i+1},\dots,v_{j-1},s,v_{j+1},\dots,v_{2m+1})
\]
where $s$ is some value from $[0,2]$. We see that the set
\[
z\times[0,1]=\big\{\,(v_1,\dots v_{i-1},t,v_{i+1},\dots,v_{j-1},s,v_{j+1},\dots,v_{2m+1})\mid t\in [0,1]\,\big\}
\]
also belongs to the same square $C$ above. We conclude that the product 
\[
H\times[0,1]=\big\{\,(x_1,\dots x_{i-1},t,x_{i+1},\dots,x_n)\mid t\in [0,1]\,\big\}
\]
is a union of $2$--cells of $\widehat K_m$. This defines a combinatorial map $H\times [0,1]\to \widehat K_m$ (actually, an embedding) such that $H\times\{\frac12\}$ is identified with $H$ itself.

A similar reasoning applies if $e_\alpha=e_1$, $t=\frac32$ (we parametrize $H\times[0,1]$ by $t \in [1,2]$).

This proves that every hyperplane of $\widehat K_m$ is two-sided.
\end{proof}

\medskip
Unfortunately, cube subcomplexes of a Cartesian product of three or more graphs can have inter-osculating hyperplanes, as the example in the Figure~\ref{fig:iosc} shows. The $2$--complex in Figure~\ref{fig:iosc} is a subcomplex of the product of two segments of length one and a segment of length two.

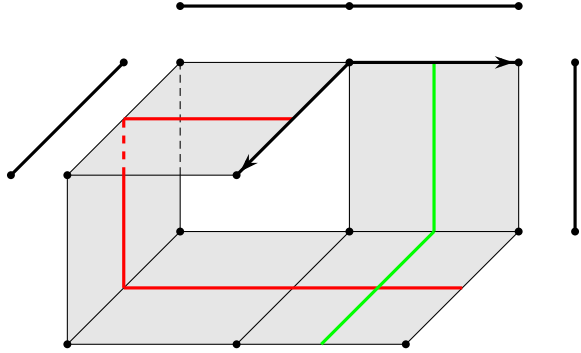
\begin{figure}[h!]
\begin{tikzpicture}[scale=0.75]

\begin{scope}[color=black!10]
\fill (5,5)--(2,5)--(0,3)--(0,0)--(6,0)--(8,2)--(2,2)--(2,3)--(3,3)--(5,5);
\fill (5,5)--(8,5)--(8,2)--(5,2)--(5,5);
\end{scope}

\fill (0,0) circle (2pt);
\fill (0,3) circle (2pt);
\fill (3,0) circle (2pt);
\fill (3,3) circle (2pt);
\fill (2,2) circle (2pt);
\fill (2,5) circle (2pt);
\fill (5,2) circle (2pt);
\fill (5,5) circle (2pt);
\fill (6,0) circle (2pt);
\fill (8,2) circle (2pt);
\fill (8,5) circle (2pt);

\draw (0,0)--(6,0);
\draw (2,2)--(8,2);
\draw (2,5)--(5,5);
\draw (0,0)--(0,3);
\draw (0,0)--(2,2);
\draw (0,3)--(2,5);
\draw (0,3)--(3,3);
\draw (3,0)--(5,2);
\draw (5,2)--(5,5);
\draw (6,0)--(8,2);
\draw (8,2)--(8,5);
\draw (2,3)--(2,2);
\draw[densely dashed] (2,3)--(2,5);

\begin{scope}[red,very thick]
\draw (4,4)--(1,4);
\draw (1,1)--(1,3);
\draw[dashed] (1,3)--(1,4);
\draw (1,1)--(7,1);
\end{scope}

\begin{scope}[green,very thick]
\draw (4.5,0)--(6.5,2);
\draw (6.5,2)--(6.5,5);
\end{scope}

\begin{scope}[]
\draw[->-=0.97,very thick] (5,5)--(3,3);
\draw[->-=0.97,very thick] (5,5)--(8,5);
\end{scope}

\begin{scope}[very thick]
\fill (-1,3) circle (2pt);
\fill (1,5) circle (2pt);
\fill (2,6) circle (2pt);
\fill (5,6) circle (2pt);
\fill (8,6) circle (2pt);
\fill (9,2) circle (2pt);
\fill (9,5) circle (2pt);
\draw (-1,3)--(1,5);
\draw (2,6)--(5,6);
\draw (5,6)--(8,6);
\draw (9,2)--(9,5);
\end{scope}

\end{tikzpicture}
\caption{A subcomplex in a product of three graphs with inter-osculating hyperplanes.\label{fig:iosc}}
\end{figure}

However, Haglund and Wise have proved in~\cite[Th.~5.7]{HW} that in the case when the square complex is a so-called VH-complex, the absence of the first three hyperplane pathologies guarantees the existence of a finite special cover.

\begin{defi} (VH-complex)
A simple square complex is called a \emph{VH-complex} if its edges are divided into two disjoint classes: \emph{vertical} and \emph{horizontal}, such that the attaching map of each square is of the form $vhv'h'$ where $v,v'$ are vertical and $h,h'$ are horizontal edges.
\end{defi}

\begin{prop}\label{prop:vh}
For all \emph{even} integers $m$, the complexes $K_{m,m}$ and $\widehat K_m$ are VH-complexes. Hence there exist a finite special cover $\overbar K_m\to \widehat K_m$.
\end{prop}
\begin{proof}
From the LOG definition (see section~\ref{sect:log}) of groups $G_{m,m}$ we observe that, for the even integers $m$, the odd-indexed and the even-indexed generators form two classes $V$ and $H$ (`vertical' and `horizontal') such that all relators of $G_{m,m}$ have the form: $v_1^h=v_2$ or $h_1^v=h_2$ for some $v,v_1,v_2\in V$, $h,h_1,h_2\in H$. This implies that the complex $K_{m,m}$ is a VH-complex.

The complex $\widehat K_m$, being a finite cover of $K_{m,m}$, inherits the structure of a VH-complex from $K_{m,m}$. Indeed, the preimages of the vertical and horizontal edges in $K_{m,m}$ under the covering map $p\colon\widehat K_m\to K_{m,m}$ form two disjoint classes $\widehat V=p^{-1}(V)$ and $\widehat H=p^{-1}(H)$, and all edges of $\widehat K_m$ are contained in $\widehat V\sqcup\widehat H$. The link of every vertex of $K_{m,m}$ is a bipartite graph corresponding to parts $V$ and $H$, and links of vertices are mapped isomorphically under covering maps. Thus all links of vertices in $\widehat K_m$ are bipartite with respect to parts $\widehat V$ and $\widehat H$. Therefore all $2$--cells of $\widehat K_m$ have boundaries of the form $v_1h_1v_2h_2$ with $v_1,v_2\in \widehat V$, $h_1,h_2\in\widehat H$. By Proposition~\ref{prop:clean}, hyperplanes of $\widehat K_m$ have no self-intersections and no self-osculations. Hence, by Theorem~5.7 in~\cite{HW}, there exist a special cube complex $\overbar K_m$ and a finite cover $\overbar K_m\to \widehat K_m$.
\end{proof}

\section{Proof of Main Theorems\label{sect:proof_main_thm}}

Now we are ready to prove our main results.

\begin{thma} For each positive \emph{even} integer $m$ there exist virtually special free-by-cyclic groups $G_{m,m}\cong F_{2m}\rtimes_\phi\Z$ with growth function $\gr_\phi(n)\sim n^m$ and $G_{m,m-1}\cong F_{2m-1}\rtimes_{\phi'}\Z$ with growth function $\gr_{\phi'}(n)\sim n^{m-1}$.
\end{thma}
\begin{proof}
We have seen in Proposition~\ref{prop:vh} that for each even integer $m>0$, there exist a special cover $\overbar K_m\to K_{m,m}$ for the presentation complex $K_{m,m}$ of the group $G_{m,m}\cong F_{2m}\rtimes_\phi\Z$. In Propositions~\ref{prop1} and~\ref{prop2} of section~\ref{sect:growth} we show that the growth function for $\phi=\phi_{m,m}$ is $\sim n^m$. This proves the first part of Theorem~A.

For the second part, recall that $G_{m,m-1}=F_{2m-1}\rtimes_{\phi'}\Z$, where $\phi'=\phi_{m,m-1}$ is the restriction of $\phi$ on the free subgroup on the first $2m-1$ generators. By construction, the presentation complex $K_{m,m-1}$ for $G_{m,m-1}$ is a subcomplex of $K_{m,m}$, and is actually obtained from $K_{m,m}$ by deleting the loop corresponding to the last generator $a_{2m+1}$ and also the single open $2$--cell adjacent to $a_{2m+1}$ (i.e. which have $a_{2m+1}$ as one of their sides).

Let $\bar p\colon \overbar K_m\to K_{m,m}$ be the special cover of $K_{m,m}$ from Proposition~\ref{prop:vh}. Consider a square subcomplex $\overbar K_m' \subset \overbar K_m$ which is obtained by:
\begin{enumerate}
\item deleting all $1$--cells of $\overbar K_m$ which map under $\bar p$ onto the loop labeled $a_{2m+1}$ in $K_{m,m}$;
\item deleting all open $2$--cells of $\overbar K_m$ which have $1$--cells from (1) as one of their sides;
\item taking a connected component of the resulting complex.
\end{enumerate}

We claim that $\bar p\colon\overbar K_m'\to K_{m,m-1}$ is a finite special cover of $K_{m,m-1}$. 

Indeed, by construction, $\bar p(\overbar K_m')$ lies in $K_{m,m-1}$. The hyperplanes of $\overbar K_m'$ are two-sided, do not self-intersect and do not self-osculate, since they are subcomplexes of the corresponding hyperplanes in the special complex $\overbar K_m$.

To see that the complex $\overbar K_m'$ has no inter-osculating hyperplanes, observe that in steps (1), (2) above we deleted only the hyperplanes which are dual to the $1$--cells corresponding to the last generator $a_{2m+1}$. This doesn't change the absence of inter-osculation of the remaining hyperplanes of $\overbar K_m$. Therefore, the hyperplanes in $\overbar K_m'$ do not inter-osculate either, and $\overbar K_m'$ is special.

Again, that the growth of $\phi'$ is $\sim n^{m-1}$ is shown in Propositions~\ref{prop1} and~\ref{prop2} in section~\ref{sect:growth}, since $\phi'=\phi_{m,m-1}$.
\end{proof}

\begin{cora} For each positive integer $k$ there exist a right-angled Artin group containing a free-by-cyclic subgroup whose monodromy automorphism has growth function $\sim n^k$.
\end{cora}
\begin{proof}
In Theorem~A we have proved that for any positive integer $k$ (where $k=m$ or $m-1$ for arbitrary even $m$) there exists a free-by-cyclic group $G=F\rtimes_\psi\Z$ with $\gr_\psi(n)\sim n^k$, such that some finite index subgroup $H\le G$ is isomorphic to a fundamental group of a special cube complex. By D.~Wise's celebrated result (see Corollary~\ref{cor:spec_raag}), there exists a right-angled Artin group $A(\Delta)$ with $H$ isomorphic to a subgroup of $A(\Delta)$.

Let's prove that $H$ is free-by-cyclic itself. Indeed, we have a commutative diagram:
\[
\begin{tikzcd}
1\arrow[r] & F\arrow[r] & F\rtimes_\psi\Z \arrow[r,"\pi"] &\Z \arrow[r] &1\\
1\arrow[r] & N \arrow[r]\arrow[u, hook] & H \arrow[u, hook]\arrow[r] & \ell\Z\arrow[u,hook]\arrow[r]&1
\end{tikzcd}
\]
Here $N=H\cap F$ and $\ell\Z$ is the image of $H$ under $\pi$. Since $H$ has finite index in $G$, $\ell\ne0$. Hence the subgroup $N$ is invariant under $\psi^\ell$ and is a free group. Since $F\triangleleft G$, $FH$ is a subgroup of $G$, and
\[
|F:N|=|F:H\cap F|=|FH:H|\le |G:H| < \infty.
\]
Therefore, $N$ is a finitely generated free group, and $H\cong N\rtimes_{\psi^\ell}\Z$, a free-by-cyclic group.

Parts (ii) and (iii) of Proposition~\ref{prop:growth} tell us now that 
\[
\gr_{\psi^\ell|_{N}}(n)\sim\gr_{\psi}(n)\sim n^k.\qedhere
\]
\end{proof}

\begin{thmb} For each positive integer $k$ there exists a right-angled Artin group which contains a finitely presented subgroup with Dehn function $\simeq n^k$.
\end{thmb}
\begin{proof}
In Theorem~A we proved that, for all even integers $m$, the free-by-cyclic group $G_{m,m}=F_{2m}\rtimes_\phi\Z$ (resp. $G_{m,m-1}=F_{2m-1}\rtimes_{\phi'}\Z$) has the following properties: $\gr_\phi(n)\sim n^m$ (resp. $\gr_{\phi'}(n)\sim n^{m-1}$), and it contains a finite index subgroup $H$ which embeds into a right-angled Artin group. 
In Corollary~A we showed that $H$ is itself free-by-cyclic: $H\cong N\rtimes_{\phi^\ell}\Z$ (resp. $H\cong N\rtimes_{{\phi'}^\ell}\Z$) for some finite index subgroup $N\le F_{2m}$ (resp. $N\le F_{2m-1}$) with the monodromy automorphism being $\phi^\ell$ (resp. $(\phi')^\ell$) for some $\ell>0$.

We claim that \textit{the Bieri double of $H$, $\Gamma(H)=H\Asterisk_FH$, has Dehn function $\delta_\Gamma(n)$ $\simeq$ equivalent to $\gr_{\phi}(n)\cdot n^2$ (resp. $\gr_{\phi'}(n)\cdot n^2$), and itself embeds into a RAAG.}

We now prove this claim for the case of subgroup $H\le G_{m,m}$, the monodromy automorphism $\phi^\ell$, and $k=m$, and notice that the case of $H\le G_{m,m-1}$, the monodromy automorphism ${\phi'}^\ell$, and $k=m-1$, is proved in a similar fashion.

The upper bound: $\delta_\Gamma(n)\preceqq n^{k+2}$ is established as follows. By Propositions~\ref{prop1}, and \ref{prop2}, $\gr_\phi(n)\sim n^k$ and $\gr_{\phi^{-1}}(n)\sim n^k$. Proposition~\ref{prop:growth}(ii) implies that $\gr_{\phi^\ell}(n)\sim n^k$ and $\gr_{(\phi^\ell)^{-1}}(n)=\gr_{(\phi^{-1})^\ell}(n)\sim n^k$, and so the upper bound follows from Proposition~\ref{prop:upper}.

The lower bound: $n\cdot \max_{\|b\|\le n,\,b\in N}\|\phi^{\ell n}(b)\|\preceqq \delta_\Gamma(n)$ was given in Proposition~\ref{prop:lower}. If we show that $\max_{\|b\|\le n,\,b\in N}\|\phi^{\ell n}(b)\|\suq n^{k+1}$, it will follow, in view of the above, that $\delta_\Gamma(n)\simeq n^{k+2}$.

We prove in section~\ref{sect:growth} that in the group $G_{m,m}=F_{2m}\rtimes_\phi\Z$, containing $H$, the maximum in the definition of the growth functions $\gr_{\phi}(n)$ and $\gr_{\phi^{ab}}(n)$ is achieved at the generator $B_m$ (see Corollary~\ref{cor:cert}): $\big\|\phi^n(B_m)\big\|\sim \big|(\phi_{m,k}^{ab})^n(\bar B_k)\big|_1 \sim n^m=n^k$. (Here the bar over an element of $F_{2m}$ denotes its image in the abelianization of $F_{2m}$.)

Since the subgroup $N$ is of finite index in $F_{2m}$, there exists an integer $p>0$ such that $B_m^p\in N$. Then we have: 
\[
\max_{\substack{\|b\|\le n\\ b\in N}}\big\|\phi^{\ell n}(b)\big\|\ge 
\big\|\phi^{\ell n}(B_m^{pn})\big\|\ge 
pn\cdot \big|(\phi^{ab})^{\ell n}(\bar{B}_m)\big|_1\su pn(\ell n)^m\sim n^{m+1}.
\]
Since $k=m$, this proves that $\delta_\Gamma(n)\simeq n^{k+2}$.

To prove that $\Gamma$ embeds into a RAAG, consider a homomorphism $\mu\colon\Gamma\to H\times F(u,v)$, where $F(u,v)$ is a free group of rank $2$ on free generators $u$, $v$, from~\cite[Rem.~3.7(iii)]{BriPoly}, which is described as follows. Let $N=F(x_1,\dots,x_q)$ be freely generated by elements $x_1,\dots,x_q$. Denote the generators of $\Z$-factors of three copies of $H$ as $s$, $t$ and $\tau$. Also denote for brevity $\psi=\phi^\ell$. Then $\Gamma=\big(F(x_1,\dots,x_q)\rtimes_\psi\la s\ra\big)\Asterisk_{F(x_1,\dots,x_q)}\big(F(x_1,\dots,x_q)\rtimes_\psi\la t\ra\big)$, and $H\times F(u,v)=\big(F(x_1,\dots,x_q)\rtimes_\psi\la \tau\ra\big) \times F(u,v)$. Define $\mu$ on the generators as follows:
\[
x_i\mapsto x_i,\qquad s\mapsto \tau u,\qquad t\mapsto \tau v.
\]
We check at once that $\mu$ is a homomorphism, and it is easily proved using the normal forms of elements in free amalgamated products, that $\mu$ is injective.

Thus, if we denote the right-angled Artin group containing $H$ as $A(\Delta)$, for some graph $\Delta$, then $\Gamma\subset H\times F_2$ is a subgroup of $A(\Delta)\times F_2$, which is itself a RAAG (corresponding to the graph join of $\Delta$ and the empty graph on two vertices).
\end{proof}

\begin{rem}
Following the proof of Proposition~\ref{prop:lower} given in~\cite[Lemma 1.5]{BriPoly} (see also~\cite[Proposition~7.2.2]{BriChap}), we can exhibit an explicit sequence of words $w_n=[(st^{-1})^n,t^{\ell n}B_k^{pn}t^{-\ell n}]$ (where $k=m$ or $m-1$, as above) which realize the lower bound for the Dehn function $\delta_\Gamma(n)$. To understand what van Kampen diagrams for these words look like, the reader is referred to the proof of~\cite[Lemma 1.5]{BriPoly}. In the notation of~\cite{BriPoly}, $\beta=B_k^{pn}$, $t_1=t$, $t_2=s$.
\end{rem}

\section{Growth of $\phi$ and $\phi^{-1}$\label{sect:growth}}

In Proposition~\ref{prop:phi} we have shown that, for all $1\le k\le m$, $G_{m,k}=F_{m+k}\rtimes_{\phi_{m,k}}\Z$, where $F_{m+k}$ is a free group with generators $A_1,\dots,A_m,B_1,\dots,B_{k}$. For convenience, in what follows we adopt the notation:
\[
\phi=\phi_{m,m}
\]
and use the fact that $\phi_{m,k}$ is the restriction of $\phi$ on the first $m+k$ generators. The goal of this section is to prove that $\phi_{m,k}$ and $(\phi_{m,k})^{-1}$ have growth $\sim n^{k}$. 
In particular, $\gr_\phi(n)\sim\gr_{\phi^{-1}}(n)\sim n^m$.

Throughout this section, $\|.\|$ will denote the word length in $F_{2m}$ with respect to the system of free generators $\{A_i,B_j\}$.

Recall (see Proposition~\ref{prop:phi}) that the automorphism $\phi$ is given by the formulas (where the overbar denotes the inverse):
\begin{align*}\tag{1}\label{eq1}
\quad\phi=\phi_{m,m}\colon\quad & A_1\longmapsto A_1\\
& A_2\longmapsto A_1\,(A_2)\,\ov{A}_1\\
& A_3\longmapsto A_1A_2\,(A_3)\,\ov{A}_2\ov{A}_1\\
&\dots\\
& A_{m}\longmapsto A_1A_2\ldots A_{m-1}\,(A_m)\,\ov{A}_{m-1}\ldots \ov{A}_1\\[1em]
&B_1\longmapsto A_1A_2\ldots A_m\,(B_1)\\
&B_2\longmapsto A_1A_2\ldots A_m\,(B_1B_2)\,\ov{A}_1\\
&B_3\longmapsto A_1A_2\ldots A_m\,(B_1B_2B_3)\,\ov{A}_2\ov{A}_1\\
&\dots\\
&B_{m}\longmapsto A_1A_2\ldots A_m\,(B_1B_2\ldots B_{m})\,\ov{A}_{m-1}\ov{A}_{m-2}\ldots\ov{A}_2\ov{A}_1.
\end{align*}

\bigskip
\subsection{Upper bounds for the growth of $\phi$, $\phi^{-1}$}

\begin{prop}\label{prop1}
For the automorphism $\phi_{m,k}$ we have:

\[
\gr_{\phi_{m,k}}(n)\pr n^k\text{\quad and \quad}\gr_{(\phi_{m,k})^{-1}}(n)\pr n^k.
\]
\end{prop}

We will need few basic lemmas.

\begin{lem}\label{lem1} For $i=1,\dots,m$,
\[
\pn(A_i)=A_1^nA_2^n\dots A_{i-1}^n\cdot A_i\cdot \ov{A}_{i-1}^n\dots \ov{A}_2^n\ov{A}_1^n.
\]
\end{lem}
\begin{proof*}
We prove the statement by induction on $n$, observing that it is true for $n=0,1$:
\begin{multline*}
\phi^{n+1}(A_i)=\phi(\phi^n(A_i))=\phi(A_1^nA_2^n\dots A_{i-1}^n)\cdot\phi(A_i)\cdot\phi(\ov{A}_{i-1}^n\dots \ov{A}_2^n\ov{A}_1^n)\\
=\phi(A_1^n)\phi(A_2^n)\phi(A_3^n)\dots \phi(A_{i-1}^n)\cdot\phi(A_i)\cdot\phi(\ov{A}_{i-1}^n)\dots\phi(\ov{A}_3^n)\phi(\ov{A}_2^n)\phi(\ov{A}_1^n)\\
=(A_1^n)(A_1A_2^n\ov{A}_1)(A_1A_2A_3^n\ov{A}_2\ov{A}_1)\dots(A_1\dots A_{i-2}A_{i-1}^n\ov{A}_{i-2}\dots \ov{A}_1)\\
\times(A_1\dots A_{i-1}A_{i}\ov{A}_{i-1}\dots \ov{A}_1)\times(A_1\dots A_{i-2}\ov{A}_{i-1}^n\ov{A}_{i-2}\dots \ov{A}_1)\dots(A_1A_2\ov{A}_3^n\ov{A}_2\ov{A}_1)\\
\times(A_1\ov{A}_2^n\ov{A}_1)(\ov{A}_1^n)=
A_1^{n+1}A_2^{n+1}\dots A_{i-1}^{n+1}\cdot A_i\cdot\ov{A}_{i-1}^{n+1}\dots \ov{A}_2^{n+1}\ov{A}_1^{n+1}.\tag*{$\qed$}
\end{multline*}
\end{proof*}

\begin{cor}\label{cor1}For $i=1,\dots,m$, 
\[\pushQED{\qed}
\|\pn(A_i)\|=2(i-1)n+1.\qedhere\popQED
\]
\end{cor}

\begin{lem}\label{lem2}
$\pn(B_1)=A_1^nA_2^n\dots A_{m}^n\cdot B_1.$
\end{lem}
\begin{proof*}
The statement is true for $n=0,1$. By induction,
\begin{multline*}
\phi^{n+1}(B_1)=\phi(\phi^n(B_1))=\phi(A_1^nA_2^n\dots A_{m}^n)\cdot\phi(B_1)\\
=(A_1^n)(A_1A_2^n\ov{A}_1)(A_1A_2A_3^n\ov{A}_2\ov{A}_1)\dots(A_1\dots A_{m-1}A_{m}^n\ov{A}_{m-1}\dots \ov{A}_1)\\ 
\times(A_1A_2\dots A_{m}\cdot B_1)=A_1^{n+1}A_2^{n+1}\dots A_{m}^{n+1}\cdot B_1.\tag*{$\qed$}
\end{multline*}
\end{proof*}

\begin{cor}\label{cor2}
$\|\pn(B_1)\|=mn+1.$\qed
\end{cor}

\begin{lem}\label{lem3}
$\pn(A_1A_2\dots A_{m})=A_1^{n+1}A_2^{n+1}\dots A_{m-1}^{n+1}\cdot A_{m}\cdot\ov{A}_{m-1}^n\dots \ov{A}_2^n\ov{A}_1^n.$
\end{lem}
\begin{proof*}
We do induction on $n$, the case $n=0$ being evident:
\begin{multline*}
\phi^{n+1}(A_1A_2\dots A_{m})=\phi(\phi^n(A_1A_2\dots A_{m}))=\phi(A_1^{n+1})\phi(A_2^{n+1})\dots \phi(A_{m-1}^{n+1})\\
\times\phi(A_m)\cdot\phi(\ov{A}_{m-1}^{n})\dots\phi(\ov{A}_2^{n})\phi(\ov{A}_1^{n})=
(A_1^{n+1})(A_1A_2^{n+1}\ov{A}_1)(A_1A_2A_3^{n+1}\ov{A}_2\ov{A}_1)\dots\\
\times(A_1\dots A_{m-2}A_{m-1}^{n+1}\ov{A}_{m-2}\dots \ov{A}_1)\cdot
(A_1\dots A_{m-1}A_{m}\ov{A}_{m-1}\dots \ov{A}_1)\\
\times(A_1\dots A_{m-2}\ov{A}_{m-1}^n\ov{A}_{m-2}\dots \ov{A}_1)\dots(A_1A_2\ov{A}_3^n\ov{A}_2\ov{A}_1)\cdot(A_1\ov{A}_2^n\ov{A}_1)(\ov{A}_1^n)\\
=A_1^{n+2}A_2^{n+2}\dots A_{m-1}^{n+2}\cdot A_m\cdot\ov{A}_{m-1}^{n+1}\dots \ov{A}_2^{n+1}\ov{A}_1^{n+1}.\tag*{$\qed$}
\end{multline*}
\end{proof*}

\begin{lem}\label{lem4}
$\|\pn(B_2)\|=\frac m2n^2+(\frac m2+2)n+1.$
\end{lem}
\begin{proof*}
One observes that 
\[
\phi(B_2)=\phi(B_1)\cdot B_2\cdot \ov{A}_1.
\]
This gives by induction in view of Lemma~\ref{lem2}:
\begin{multline*}
\pn(B_2)=\pn(B_1)\phi^{n-1}(B_1)\dots \phi(B_1)\cdot B_1\cdot \ov{A}_1^n\\
=(A_1^nA_2^n\dots A_{m}^n B_1)\cdot(A_1^{n-1}A_2^{n-1}\dots A_{m}^{n-1} B_1)\dots 
(A_1A_2\dots A_{m} B_1)\cdot B_1\cdot \ov{A}_1^n.
\end{multline*}
Hence, 
\begin{multline*}
\|\pn(B_2)\|=[mn+1]+[m(n-1)+1]+\ldots+[m+1]+1+n\\
=m\tfrac{n(n+1)}{2}+2n+1=\tfrac m2n^2+\left(\tfrac m2+2\right)n+1.\tag*{$\qed$}
\end{multline*}
\end{proof*}

\begin{claim}\label{claim1}For $k=1,\dots,m$,
\[
\|\pn(B_k)\|\pr n^k.
\]
\end{claim}
\begin{proof} From the formulas~\eqref{eq1} we get for all $k\geq2$,
\[
\phi(B_{k+1})=\phi(B_k)\cdot(A_1\dots A_{k-1})\cdot B_{k+1}\cdot\phi(\ov{A}_k\dots\ov{A}_1).
\]
Therefore,
\begin{equation*}\tag{2}\label{eq2}
\pn(B_{k+1})=\pn(B_k)\cdot\phi^{n-1}(A_1\dots A_{k-1})\cdot\phi^{n-1}(B_{k+1})\cdot\phi^{n-1}(A_1\dots A_k)^{-1}.
\end{equation*}
Lemma~\ref{lem3} gives:
\begin{align*}
\|\phi^{n-1}(A_1\dots A_{k-1})\|&=2(k-2)(n-1)+(k-1),\\
\|\phi^{n-1}(A_1\dots A_{k})^{-1}\|&=2(k-1)(n-1)+k,
\end{align*}
so that the total length of $\pn(B_{k+1})$ is bounded above by 
\[
\|\pn(B_k)\|+\|\phi^{n-1}(B_{k+1})\|+[(4k-6)n-(2k-5)].
\]
Now if we denote
\[
f(k,n):=\|\pn(B_k)\|,
\]
we will have
\begin{itemize}
\item $f(k,0)=1$;
\item $f(1,n)=mn+1$, by Corollary~\ref{cor2};
\item $f(2,n)=\frac m2n^2+(\frac m2+2)+1$, by Lemma~\ref{lem4};
\end{itemize}
and for $k\geq 2$,
\[
f(k+1,n)\leq f(k,n)+f(k+1,n-1)+[(4k-6)n-(2k-5)].
\]
We have an inequality here (instead of an equality) because in the formula~\eqref{eq2} there can be some cancellations. Let's define another function $g(k,n)$ as follows:
\begin{itemize}
\item $g(k,0)=1$;
\item $g(1,n)=f(1,n)=mn+1$,
\item $g(2,n)=f(2,n)=\frac m2n^2+(\frac m2+2)+1$,
\item $g(k+1,n)=g(k,n)+g(k+1,n-1)+[(4k-6)n-(2k-5)]$, for $k\geq2$.
\end{itemize}

Obviously, $g$ is well-defined in a recurrent fashion. An easy induction shows that
\[
f(k,n)\leq g(k,n)\text{\quad for all $k\geq 1$, $n\geq0$,}
\]
so that $g(k,n)$ gives an upper bound for the growth of $\|\pn(B_k)\|$.

To estimate the order of growth of $g(k,n)$, let's look at finite differences in $n$:
\[
g(k+1,n)-g(k+1,n-1)=g(k,n)+[(4k-6)n-(2k-5)]
\]
so if we assume by induction that $g(k,n)$ is a polynomial in $n$ of degree $k$, then we conclude that $g(k+1,n)$ is a polynomial of degree $k+1$ in $n$.

Since this assumption is true for $k=1$ and $2$, this proves Claim~\ref{claim1}.
\end{proof}

\bigskip
Now we establish a similar bound for $\phii$. One could use the train-track machinery (along the lines of~\cite[Th.~0.4]{Pig}) to prove that the growth of $\phii$ is the same as the growth of $\phi$, when it is polynomial, but we give here a simple direct proof.

One easily checks that the inverse automorphism $\phii$ acts as follows:
\begin{align*}\tag{3}\label{eq3}
\quad\phii\colon\quad & A_1\longmapsto A_1\\
& A_2\longmapsto \ov{A}_1\,(A_2)\,A_1\\
& A_3\longmapsto \ov{A}_1\ov{A}_2\,(A_3)\,A_2A_1\\
&\dots\\
& A_{m}\longmapsto \ov{A}_1\ov{A}_2\dots\ov{A}_{m-1}\,(A_m)\,{A}_{m-1}\dots{A}_2{A}_1\\[1em]
&B_1\longmapsto \ov{A}_1\ov{A}_2\dots\ov{A}_{m-1}\ov{A}_m\cdot B_1\\
&B_2\longmapsto (\ov{B}_1B_2)A_1\\
&B_3\longmapsto \ov{A}_1\,(\ov{B}_2B_3)\,A_2A_1\\
&\dots\\
&B_{m}\longmapsto \ov{A}_1\ov{A}_2\ldots\ov{A}_{m-2}\,(\ov{B}_{m-1}B_m)\,A_{m-1}\ldots{A}_2{A}_1.
\end{align*}

\begin{lem}\label{lem5}For $i=1,\dots,m$,
\[
\|\pu(A_i)\|=\|\pn(A_i)\|=2(i-1)n+1.
\]
\end{lem}
\begin{proof}
Define an automorphism $\iota\colon H\longrightarrow H$ of the subgroup $H=\la A_1, \dots, A_m\ra$ given by $\iota\colon A_j\mapsto\ov{A}_j$, $j=1,\dots,m$. One easily checks that
\[
\phii|_H=\iota\circ(\phi|_H)\circ\iota^{-1},
\]
therefore $\|\pu(A_i)\|=\|\pn(A_i)\|$ and the result follows from Corollary~\ref{cor1}.
\end{proof}

\begin{claim}\label{claim2}For $k=1,\dots,m$,
\[
\|\pu(B_k)\|\pr n^k.
\]
\end{claim}
\begin{proof}
Denote for any $i\geq0$, $k\geq2$:
\begin{align*}
T_{1,i}&:=B_1,\\
T_{k,i}&:=B_k\cdot A_{k-1}^iA_{k-2}^i\dots A_2^iA_1^i,\\
S_i&:=\ov{A}_1^i\ov{A}_2^i\dots\ov{A}_{m-1}^i\ov{A}_m^i.
\end{align*}
In this notation, the action of $\phii$ can be written as follows:
\begin{align*}\tag{4}\label{eq4}
\phii(T_{k,i})&=\ov{T_{k-1,1}}\cdot T_{k,i+2},\\
\phii(S_i\cdot T_{1,1})&=S_{i+1}\cdot T_{1,1}.
\end{align*}
The first relation is obvious, and the second one follows by easy induction.

\begin{lem}\label{lem6}
$\|\pu(B_1)\|=mn+1$.
\end{lem}
\begin{proof*}
Indeed, 
\begin{multline*}
\pu(B_1)=\pu(S_0\cdot T_{1,1})=\phi^{-(n-1)}(S_1\cdot T_{1,1})=\phi^{-(n-2)}(S_2\cdot T_{1,1})=\dots\\
=\phi^{-2}(S_{n-2}\cdot T_{1,1})=\phii(S_{n-1}\cdot T_{1,1})=S_n\cdot T_{1,1}.\tag*{$\qed$}
\end{multline*}
\end{proof*}

Define for any $k\geq1$, $i\geq0$, $n\geq0$ a function $f(k,i,n)$ as follows:
\begin{itemize}
\item $f(1,i,n)=mn+1$;
\item $f(k,i,n)=\|\pu(T_{k,i})\|$, for $k\geq2$.
\end{itemize}
Then relations~\eqref{eq4} imply
\[
f(k,i,n+1)\leq f(k-1,1,n)+f(k,i+1,n),
\]
where we have an inequality (but not an equality) because of possible cancellations in the reduced expression for $\pu(T_{k,i})$.

To obtain an upper bound on $f(k,i,n)$ we introduce a function $g(k,i,n)$ defined recurrently as follows:
\begin{itemize}
\item $g(k,i,0)=\|T_{k,i}\|=(k-1)i+1$, for all $k\geq1$, $i\geq0$;
\item $g(1,i,n)=\|\pu(B_1)\|=mn+1$, for all $i\geq0$, $n\geq0$;
\item $g(k,i,n+1)=g(k-1,1,n)+g(k,i+1,n)$, for all $k\geq2$, $i\geq0$, $n\geq0$.
\end{itemize}
These formulas define $g(k,i,n)$ recurrently for all values of $k\geq1$, $i\geq0$, $n\geq0$. Indeed, one proceeds by layers numbered by $n$, with the case $n=0$ given by the first formula, and the case of arbitrary $n$ given by the third one, which is valid for $k\geq2$. The remaining case $k=1$ is given by the second formula.

Clearly, $f(k,i,n)\leq g(k,i,n)$ for all $k\geq1$, $i\geq0$, $n\geq0$, so that the function $g$ can be used to establish the upper bound for $\|\pu(B_k)\|$:
\[
\|\pu(B_k)\|=\|\pu(T_{k,0})\|=f(k,0,n)\leq g(k,0,n).
\]

To estimate the growth of $g(k,i,n)$, consider the finite difference $g(k,i,n+1)-g(k,i,n)$. Applying the recurrent relation several times, we get:
\begin{align*}
g(k,i,n+1)&=g(k-1,1,n)+g(k,i+1,n)\\
&=g(k-1,1,n)+g(k-1,1,n-1)+g(k,i+2,n-1)\\
&\dots\\
=[g(k-1,1,n)&+g(k-1,1,n-1)+\dots+g(k-1,1,0)]+g(k,i+n+1,0).
\end{align*}
Similarly,
\[
g(k,i,n)=[g(k-1,1,n-1)+g(k-1,1,n-2)+\dots+g(k-1,1,0)]+g(k,i+n,0).
\]
Since by definition
\begin{align*}
g(k,i+n+1,0)&=(k-1)(i+n+1)+1,\\
g(k,i+n,0)&=(k-1)(i+n)+1,
\end{align*}
we have for all $k\geq2$, $i\geq0$, $n\geq0$:
\begin{equation*}\tag{5}\label{eq5}
g(k,i,n+1)-g(k,i,n)
=g(k-1,1,n)+(k-1).
\end{equation*}
In particular,
\begin{equation*}
g(k,1,n+1)-g(k,1,n)=g(k-1,1,n)+(k-1).
\end{equation*}
If we assume by induction on $k$ that $g(k-1,1,n)$ is a polynomial function in $n$ of degree $k-1$ (which is true for $k=2$ since $g(1,i,n)=mn+1$), then we conclude at once that $g(k,1,n)$ is a polynomial function in $n$ of degree $k$.

Now the formula~\eqref{eq5} similarly implies that $g(k,i,n)$ is a polynomial in $n$ of degree~$k$.

Therefore, $\|\pu(B_k)\|\leq g(k,0,n)\sim n^k$, which finishes the proof of Claim~\ref{claim2}.
\end{proof}

\begin{proof}[Proof of Proposition~\ref{prop1}]
According to Corollary~\ref{cor1} and Lemma~\ref{lem5}, $\|\phi^{\pm n}(A_i)\|\pr n$ for $i=1,\dots,m$, and according to Claims~\ref{claim1} and \ref{claim2}, $\|\phi^{\pm n}(B_k)\|\pr n^k$, for $k=1$, \dots, $m$. Therefore $\gr_{\phi_{m,k}}(n)\pr n^k$ and $\gr_{(\phi_{m,k})^{-1}}(n)\pr n^k$.
\end{proof}

\bigskip
\subsection{Lower bounds for the growth of $\phi$, $\phi^{-1}$}

\begin{prop}\label{prop2}
For the automorphism $\phi_{m,k}$ and $\phi^{-1}_{m,k}$, we have:
\[
\gr_{\phi_{m,k}}(n)\su n^k\text{\quad and \quad}\gr_{(\phi_{m,k})^{-1}}(n)\su n^k.
\]
\end{prop}

\begin{claim}\label{claim3}
The size of the largest Jordan block of the Jordan normal form for both $\phi_{m,k}^{ab}$, $(\phi_{m,k}^{-1})^{ab}$ is $k+1$.
\end{claim}
\begin{proof}
It is sufficient to prove the claim just for $\phi_{m,k}^{ab}$, as $(\phi_{m,k}^{-1})^{ab}=(\phi^{ab}_{m,k})^{-1}$.

By direct inspection of formulas~\eqref{eq1}, we see that $\phi_{m,k}^{ab}$ is represented by the following $(m+k)\times (m+k)$ matrix:
\[
\phi_{m,k}^{ab}=\left[
\begin{array}{c|c}
I_m & D_{mk}\\[0.5ex]
\hline\\[-2ex]
O_{km} & C_{kk}
\end{array}\right],
\]
where $I_m$ is the identity $m\times m$ matrix, $O_{km}$ is the zero $k\times m$ matrix, and $D_{mk}$ and $C_{kk}$ are $m\times k$ and $k\times k$ matrices, respectively, given by the formulas:
\[
D_{mk}=\begin{bmatrix}
1 & 1 & \ldots & 1\\
1 & 1 & \ldots & 1\\
\hdotsfor{4}\\
1 & 1 & \ldots & 1
\end{bmatrix},\qquad
C_{kk}=\begin{bmatrix}
1 & 1 & \ldots & 1\\
0 & 1 & \ldots & 1\\
\hdotsfor{4}\\
0 & 0 & \ldots & 1
\end{bmatrix}.
\]

It is known that the number of Jordan blocks of a matrix $A\in GL(m+k,\C)$ with all eigenvalues $1$ is given by the number
\[
\dim\ker(A-I_{m+k})=m+k-\rk(A-I_{m+k}),
\]
and the number of Jordan blocks of $A$ with all eigenvalues $1$ and size at least $2$ is given by
\[
\dim\ker[(A-I_{m+k})^2] - \dim\ker(A-I_{m+k})=\rk(A-I_{m+k})-\rk[(A-I_{m+k})^2].
\]

An easy computation shows that
\[
\phi_{m,k}^{ab}-I_{m+k}=\left[
\begin{array}{c|c}
O_{mm} & D_{mk}\\[0.5ex]
\hline\\[-2ex]
O_{km} & C'_{kk}
\end{array}\right],
\text{\quad where\quad}
C'_{kk}=\begin{bmatrix}
0 & 1 & 1 & \ldots & 1\\
0 & 0 & 1 & \ldots & 1\\
\hdotsfor{5}\\
0 & 0 & 0 & \ldots & 1\\
0 & 0 & 0 & \ldots & 0
\end{bmatrix},
\]
and
\[
(\phi_{m,k}^{ab}-I_{m+k})^2=\left[
\begin{array}{c|c}
O_{mm} & D'_{mk}\\[0.5ex]
\hline
\\[-2ex]
O_{km} & C''_{kk}
\end{array}\right],
\]
where
\[
D'_{mk}=\begin{bmatrix}
0 & 1 & 2 & \ldots & k-1\\
0 & 1 & 2 & \ldots & k-1\\
\hdotsfor{5}\\
0 & 1 & 2 & \ldots & k-1
\end{bmatrix},\qquad
C''_{kk}=\begin{bmatrix}
0 & 0 & 1 & 2 & \ldots & k-2\\
0 & 0 & 0 & 1 & \ldots & k-3\\
\hdotsfor{6}\\
0 & 0 & 0 & 0 & \ldots & 1\\
0 & 0 & 0 & 0 & \ldots & 0\\
0 & 0 & 0 & 0 & \ldots & 0
\end{bmatrix}
\]

Note that $\rk C'_{kk}=k-1$, $\rk C''_{kk}=k-2$, hence $\rk(\phi_{m,k}^{ab}-I_{m+k})=k$ and $\rk[(\phi_{m,k}^{ab}-I_{m+k})^2]=k-1$. Therefore, 
\begin{gather*}
\text{ the number of Jordan blocks for $\phi_{m,k}^{ab}$} = (m+k)-k = m,\\
\text{ the number of Jordan blocks of size $\geq2$ for $\phi_{m,k}^{ab}$} = k-(k-1)=1.
\end{gather*}

This means that there is only one block of size bigger than $1$, let's denote this size $c$, and there are $m-1$ blocks of size $1$. Hence, $m+k=c+(m-1)\cdot 1$ so that $c=k+1$.
\end{proof}

\begin{proof}[Proof of Proposition~\ref{prop2}]
The Proposition follows now from Corollary~\ref{cor4} and Claim~\ref{claim3}. 
\end{proof}

\subsection{Lower bounds for the growth of $\phi^{ab}$}
In the proof of Theorem~B in section~\ref{sect:proof_main_thm} we needed a certificate for the growth of the abelianization of $\phi_{m,k}$, i.e. an element of the basis that realizes the maximum in the definition of the growth functions. Now we can provide it:
\begin{cor}\label{cor:cert}
With the above notation, let $\bar B_k$ be the image of the generator $B_k$ in the abelianization of $F$. Then
\[
\big\|(\phi_{m,k})^n(B_k)\big\|\sim \big|(\phi_{m,k}^{ab})^n(\bar B_k)\big|_1 \sim \big|(\phi_{m,k}^{ab})^n(\bar B_k)\big|_\infty \sim n^k.
\]
\end{cor}
\begin{proof}
An elementary computation with the matrix from the proof of Claim~\ref{claim3} shows that $(\phi_{m,k}^{ab})^n$ is represented by the matrix with the following structure:
\[
(\phi_{m,k}^{ab})^n=\left[
\begin{array}{c|c}
I_m & P_{mk}\\[0.5ex]
\hline\\[-2ex]
O_{km} & Q_{kk}
\end{array}\right],
\]
where $I_m$ is the identity $m\times m$ matrix, $O_{km}$ is the zero $k\times m$ matrix, and $P_{mk}$ and $Q_{kk}$ are $m\times k$ and $k\times k$ matrices, respectively, given by the formulas:
\[
P_{mk}=\begin{bmatrix}
c_{1n} & c_{2n} & \ldots & c_{kn}\\
c_{1n} & c_{2n} & \ldots & c_{kn}\\
\hdotsfor{4}\\
c_{1n} & c_{2n} & \ldots & c_{kn}
\end{bmatrix},\qquad
Q_{kk}=\begin{bmatrix}
1 & c_{1n} & \ldots & c_{k-2,n} & c_{k-1,n}\\
0 & 1 & \ldots & c_{k-3,n} & c_{k-2,n}\\
\hdotsfor{5}\\
0 & 0 & \ldots & 1 & c_{1n}\\
0 & 0 & \ldots & 0 & 1
\end{bmatrix}.
\]
Here $P_{mk}$ has all rows equal to each other and $Q_{kk}$ is upper triangular with the same number on each diagonal sequence of entries parallel to the main diagonal. The numbers $c_{1n}$, $c_{2n}$, \ldots, $c_{kn}$ satisfy the following identity, which follows from the matrix multiplication rule:
\[
c_{\ell,n+1}=\sum_{i=0}^\ell c_{in},
\]
with the convention that $c_{0n}=1$. We now show by double induction on pairs $(i,n)$ that $c_{in}=\binom{n+i-1}{i}$. Indeed, this equality is true for pairs $(i,n)=(0,n)$ with arbitrary $n$, since $c_{0n}=1$ by our convention, and for $(i,n)=(i,1)$ with arbitrary $i$, since $c_{i1}=1$ in the matrix representation for $\phi_{m,k}^{ab}$ (see the proof of Claim~\ref{claim3}). Now suppose that the equality $c_{in}=\binom{n+i-1}{i}$ is already proved for all pairs $(i,n)$ with $n$ fixed and $i$ arbitrary, and for all pairs $(i,n+1)$ with $0\le i\le \ell-1$. Then for the pair $(\ell,n+1)$ we get:
\[
c_{\ell,n+1}=\sum_{i=0}^{\ell-1} c_{in} + c_{\ell,n} = c_{\ell-1,n+1}+c_{\ell,n}=\binom{n+\ell-1}{\ell-1}+\binom{n+\ell-1}{\ell}=\binom{n+\ell}{\ell},
\]
as needed.

In particular, the coefficients of the vector $(\phi_{m,k}^{ab})^n(\bar B_k)$ are: $1$, $c_{1n}$, \dots, $c_{kn}$, with $c_{kn}=\binom{n+k-1}{k}$ being a polynomial $\sim n^k$.

Since, by Claim~\ref{claim1}, $\|\phi_{m,k}^n(B_k)\|\pr n^k$, we have:
\[
n^k \su \big\|\phi_{m,k}^n(B_k)\big\| \ge \big|(\phi_{m,k}^{ab})^n(\bar B_k)\big|_1 \ge \big|(\phi_{m,k}^{ab})^n(\bar B_k)\big|_\infty \su n^k,
\]
and we conclude that 
\[
\big\|\phi_{m,k}^n(B_k)\big\|\sim \big|(\phi_{m,k}^{ab})^n(\bar B_k)\big|_1 \sim \big|(\phi_{m,k}^{ab})^n(\bar B_k\big)|_\infty \sim n^k.\qedhere
\]
\end{proof}

\begin{rem} It can be proved in a similar manner that the same elements $B_k$ and $\bar B_k$ also serve as certificates for the growth of automorphisms $\phi_{m,k}^{-1}$ and $(\phi^{ab}_{m,k})^{-1}$, respectively. However the formulas involved are more complicated, and we don't need this result for our construction.
\end{rem}

\section{Open questions}
We conclude our paper with two open questions.
\begin{que} 
Does every CAT(0) free-by-cyclic group virtually embed into a RAAG?
\end{que}

\begin{que}
Do there exist finitely presented subgroups of RAAGs whose Dehn functions are either super-exponential or sub-exponential but not polynomial?
\end{que}


\end{document}